\documentclass[12pt,reqno]{amsart}
\usepackage[utf8]{inputenc}
\usepackage[left=3.5cm,right=3.5cm,top=3cm,bottom=3cm]{geometry}
\usepackage{graphicx}

\usepackage{amsmath, amsfonts, amssymb, amsthm, mathtools}

\usepackage[colorlinks=true,
        citecolor=blue
        ]{hyperref}
\usepackage{breakurl}
\usepackage{cite}

\usepackage{enumitem}
\usepackage{multirow}

\usepackage{xifthen}

\theoremstyle{plain}
\newtheorem{prop}{Proposition}[section]
\newtheorem{thm}[prop]{Theorem}
\newtheorem{lemma}[prop]{Lemma}
\newtheorem{cor}[prop]{Corollary}
\theoremstyle{remark}
\theoremstyle{definition}
\newtheorem{rmrk}[prop]{Remark}
\newtheorem{mydef}[prop]{Definition}

\newcommand{\cat}[1][C]{\mathcal{#1}}

\newcommand{\inv}{^{-1}}
\newcommand{\sweedler}[1]{_{(#1)}}
\newcommand{\dualsymbol}{\vee}
\newcommand{\dualL}[1]{#1^{\dualsymbol}}
\newcommand{\ddualL}[1]{#1^{\dualsymbol\dualsymbol}}
\newcommand{\dualR}[1]{\prescript{\dualsymbol}{}{#1}}

\newcommand{\one}{\mathbf{1}}

\DeclareMathOperator{\Hom}{Hom}
\DeclareMathOperator{\id}{id}
\DeclareMathOperator{\End}{End}
\DeclareMathOperator{\ev}{ev}
\DeclareMathOperator{\coev}{coev}
\DeclareMathOperator{\evL}{\overset{\xleftarrow{\phantom{ev}}}{\ev}}
\DeclareMathOperator{\coevL}{\overset{\xleftarrow{\phantom{coev}}}{\coev}}
\DeclareMathOperator{\evR}{\overset{\xrightarrow{\phantom{ev}}}{\ev}}
\DeclareMathOperator{\coevR}{\overset{\xrightarrow{\phantom{coev}}}{\coev}}

\newcommand{\Vect}{{\mathsf{Vect}}}

\newcommand{\coint}{\boldsymbol{\lambda}}
\newcommand{\cointL}{\coint^l}
\newcommand{\cointR}{\coint^r}
\newcommand{\intQ}{\boldsymbol{\Lambda}}
\newcommand{\trivialMod}[1]{\prescript{}{\counit}{#1}}
\newcommand{\field}{\Bbbk}
\newcommand{\modTr}{\mathsf{t}}
\newcommand{\proj}[1][\cat]{\Proj(#1)}
\newcommand{\op}{^{\textup{op}}}
\newcommand{\cop}{^{\textup{cop}}}
\DeclareMathOperator{\Proj}{Proj}

\newcommand{\oneQ}{\boldsymbol{1}}
\newcommand{\counit}{\varepsilon}
\newcommand{\coassQ}{\Phi}
\newcommand{\invCoassQ}{\Psi}
\newcommand{\betaQ}{\boldsymbol{\beta}}
\newcommand{\alphaQ}{\boldsymbol{\alpha}}
\newcommand{\pivotQ}{{\boldsymbol{g}}}
\newcommand{\Dt}{\boldsymbol{f}}
\newcommand{\modulus}{{\boldsymbol{\gamma}}}
\newcommand{\pR}{p^R}
\newcommand{\qR}{q^R}
\newcommand{\pL}{p^L}
\newcommand{\qL}{q^L}
\newcommand{\elU}{\textsf{U}}
\newcommand{\elV}{\textsf{V}}
\newcommand{\elu}{u}
\newcommand{\cointSym}{\widehat{\coint}}
\newcommand{\cointSymR}{\cointSym^r}
\newcommand{\cointSymL}{\cointSym^l}

\newcommand{\symIntSpace}{Sym_{tr}}

\newcommand{\hpmod}[1][H]{#1\text{-pmod}}
\newcommand{\hmod}[1][H]{#1\text{-mod}}

\newcommand{\Fr}{\phi^r}
\newcommand{\Fl}{\phi^l}
\newcommand{\Gr}{\psi^r}
\newcommand{\Gl}{\psi^l}

\newcommand{\symFerm}{\mathsf{Q}}
\newcommand{\genK}{\mathsf{K}}
\newcommand{\genF}{\mathsf{f}}
\newcommand{\eQ}{\boldsymbol{e}}
\newcommand{\ribbon}{\boldsymbol{v}}
\newcommand{\drinfeldElement}{\boldsymbol{u}}
\newcommand{\basisEl}[3]{B_{\vec{#1},\vec{#2},#3}}

\numberwithin{equation}{section}
\newcommand{\ipic}[3]{\raisebox{#1\height}{\scalebox{#3}{\includegraphics{#2.pdf}}}}

\title{Modified traces for quasi-Hopf algebras}

\begin{document}
\thispagestyle{empty}

\maketitle

\begin{center} 
		Johannes Berger\,$^{a}$,
		Azat M.\ Gainutdinov\,$^{a,b}$~~and~~Ingo Runkel\,$^a$~~\footnote{Emails: 
			{\tt johannes.berger@uni-hamburg.de},
			{\tt azat.gainutdinov@lmpt.univ-tours.fr}, 
			{\tt ingo.runkel@uni-hamburg.de}
        }
	\\[1.5em]
	{\sl\small $^a$ Fachbereich Mathematik, Universit\"at Hamburg\\
	Bundesstra\ss e 55, 20146 Hamburg, Germany}
\\[0.5em]
	{\sl\small $^b$ Institut Denis Poisson, CNRS, Universit\'e de Tours, Universit\'e d'Orl\'eans,\\ Parc de Grammont, 37200 Tours, France}
\end{center}

\vspace*{3em}

\begin{abstract}
Let $H$ be a finite-dimensional unimodular pivotal quasi-Hopf algebra over a field $\field$, and let $\hmod$ be the pivotal tensor category of finite-dimensional $H$-modules. We give a bijection between left (resp.\ right) modified traces on the tensor ideal $\hpmod$ of projective modules and left (resp.\ right) cointegrals for $H$. 
The non-zero left/right modified traces are non-degenerate, and we show that non-degenerate left/right modified traces can only exist for unimodular $H$.
This generalises results
of Beliakova, Blanchet, and Gainutdinov \cite{BBG} from Hopf algebras to quasi-Hopf algebras. 
As an example we compute cointegrals and modified traces for the family of symplectic fermion quasi-Hopf algebras.
\end{abstract}

\vspace*{4em}

\tableofcontents

\newpage
\section{Introduction}
Modified traces were introduced in \cite{GPMV, GKPM1}. 
They are a generalisation of the categorical trace in a pivotal linear category.
The latter is defined on the whole category and expressed in terms of duality morphisms and the pivotal structure.
The former are only defined on a suitable tensor ideal of the pivotal category, but they have better non-degeneracy properties than the categorical trace. 

The example of interest to us is that of a pivotal and unimodular finite tensor category
$\cat$, and its tensor ideal $\proj$ of projective objects.
If $\cat$ is not semisimple, the categorical trace vanishes identically on $\proj$.
For the modified trace, the opposite happens: there exists a non-zero modified trace on $\proj$ that induces a non-degenerate pairing on Hom-spaces	\cite{CGPM,GR,GKPM2}.
This observation has important applications to link invariants \cite{GPT,BBGe},
to three-dimensional topological field theories
 \cite{DRGPM}, 
as well as to vertex operator algebras with non-semisimple representation theory in the context of a conjectural Verlinde formula 
\cite{GR1,CG,GR}.

While there are explicit constructions of modified traces \cite{GKPM1}, these can be tedious to deploy in examples.
In~\cite{BBG} a simple description of modified traces was found in the case that $\cat = \hmod$ is the category of finite-dimensional representations of a finite-dimensional pivotal  unimodular Hopf algebra $H$: there is a one-to-one correspondence between left/right modified traces and left/right cointegrals of the Hopf algebra.
In this paper we generalise this result to quasi-Hopf algebras.
Since the definition of cointegrals for quasi-Hopf algebras is more complicated than that 
for Hopf algebras  \cite{HN-integrals, BC1,BC2},
the treatment is more technical than that in~\cite{BBG}, but the method of proof is the same.
Other generalisations of \cite{BBG} have been given in \cite{Ha,FOG}.

\medskip

Let us state our main result in more detail for right cointegrals and right modified traces.
The results for the left variant are analogous and can be found in the main text.
Our conventions for quasi-Hopf algebras are given in Section~\ref{Sec:qHopf}, so we will be brief here.
Let $\field$ be a field and let $H$ be a finite-dimensional unimodular
pivotal quasi-Hopf algebra over $\field$
with pivot $\pivotQ \in H$.
For $\coint \in H^*$ let $\widehat{\coint}$ be defined by $\widehat{\coint}(h) =
\coint(\pivotQ h)$ for all $h\in H$.
Then $\coint$ is a right cointegral if and only if
\begin{align}
	\widehat{\coint}(h) \, \oneQ ~=~ 
    \left( \widehat{\coint} \otimes \pivotQ \right) \big(\qR \,\Delta(h) \,\pR\big)
\end{align}
for all $h\in H$ (Corollary~\ref{coro_symmetrized_simple_condition}).
Here, $\qR = \invCoassQ_1\otimes S\inv(\alphaQ \invCoassQ_3)\invCoassQ_2$, $\pR = \coassQ_1\otimes \coassQ_2 \betaQ S(\coassQ_3)$, $\coassQ$ is the coassociator of $H$, $\invCoassQ$ its inverse, and $\alphaQ, \betaQ \in H$ are the evaluation and coevaluation element.

A right modified trace on $\proj$ is a collection of linear maps
\begin{align}
	\{ \modTr_P : \End_{\cat}(P) \to \field \}_{P\in \proj}
\end{align}
satisfying cyclicity and compatibility with the categorical trace (or rather with the right partial categorical trace, see Section~\ref{sec:modtr-def}).
The family $\modTr_\bullet$ is uniquely determined by its value on a projective generator. 
A right modified trace provides a pairing $\cat(M,P) \times \cat(P,M) \to \field$, $(f,g) \mapsto \modTr_P(f \circ g)$, for all $P \in \proj$ and $M \in \cat$.
If all these pairings are non-degenerate, $\modTr_\bullet$ is called non-degenerate.

Let now $\cat = \hmod$, the pivotal finite tensor category of finite-dimensional left $H$-modules.
We can take the left regular module $H$ as
 projective generator.
A non-degenerate right modified trace on $\proj$ exists if and only if $H$ is unimodular (Theorem~\ref{thm:symcoint-is-modtr}\,(1)).
In this case we prove (Theorem~\ref{thm:symcoint-is-modtr}\,(2)):

\begin{thm}\label{thm:intro}
Let $H$ be a finite-dimensional pivotal unimodular quasi-Hopf algebra over a field $\field$.
There is a one-to-one correspondence between right modified traces 
	$\modTr_\bullet$
on $\mathrm{Proj}(\hmod)$ and right cointegrals 
	$\coint \in H^*$
via
$$
    \modTr_H(f) = \widehat{\coint}\big(f(\oneQ)\big) \qquad \text{for all} \quad f \in \End_H(H) \ .
$$
In particular, such traces exist and are unique up to scalar multiples. 
Every non-zero right modified trace $\modTr_\bullet$  on $\mathrm{Proj}(\hmod)$ is non-degenerate.
\end{thm}

As an illustration we consider the family $\symFerm(N,\beta)$ of symplectic fermion quasi-Hopf
	algebras
introduced in \cite{FGR2}.
Here $N \in \mathbb{Z}_{>0}$ and $\beta \in \mathbb{C}$ satisfies $\beta^4 = (-1)^N$. 
$\symFerm(N,\beta)$ is a non-semisimple factorisable ribbon quasi-Hopf algebra, and so $\symFerm(N,\beta)\text{-mod}$ is a (non-semisimple) modular tensor category.
Conjecturally, for $\beta=e^{-i \pi N/4}$ 
this modular tensor category is equivalent to the category of representations of the even part of the vertex operator super algebra of $N$ pairs of symplectic fermions \cite[Conj.\,6.8]{FGR2}.
In Section~\ref{Sec:SymFerm} we give explicit formulas for the cointegrals on $\symFerm(N,\beta)$ (left and right cointegrals coincide in this case), 
as well as for the modified trace on projective modules
	(again the left and right variant coincide).

\bigskip

This paper is organised as follows.
In Sections~\ref{sec:modtr-def} and~\ref{Sec:qHopf} we review the necessary background on modified traces and on cointegrals in quasi-Hopf algebras, respectively.
In Section~\ref{sec:proof-modTr} we prove Theorem~\ref{thm:intro}, and in Section~\ref{Sec:SymFerm} we illustrate the result in the example of the symplectic fermion quasi-Hopf algebras.

\bigskip

While we were writing the present paper, the paper \cite{SS} by Shibata and Shimizu appeared, which also
contains a proof of Theorem~\ref{thm:intro}.

\bigskip

\noindent
{\bf Acknowledgements:}
We thank Ehud Meir and Tobias Ohrmann for useful discussions.
We are grateful to the anonymous referee for helpful suggestions and for providing us with
an improved version of Lemma \ref{prop:NewVersion3.7}.
AMG thanks the CNRS and ANR project JCJC ANR-18-CE40-0001 for support.
JB is supported by the Research Training Group 1670 of the DFG.

\bigskip

Throughout this paper we fix a field $\field$.

\section{Modified traces}\label{sec:modtr-def}

In this section we review the definition of modified traces from 
\cite{GKP-lastaddition,GPMV, GKPM1}
and recall some of their properties.
Throughout this section let $\cat$ be a pivotal finite tensor category over $\field$. 

\medskip

We denote the pivotal structure by $\delta_V : V \to V^{\vee\vee}$ and 
	will choose right duals and left duals to be identical as objects.
We write
\begin{align}
	\evL_V : \dualL{V} \otimes V \to \one
\quad , \qquad & 
\coevL_V: \one \to V \otimes \dualL{V} 
\nonumber	\\ 
	\evR_V : V\otimes \dualL{V} \to \one
\quad , \qquad & 
\coevR_V: \one \to \dualL{V} \otimes V \ ,
\end{align}
for the left ($\leftarrow$) and right ($\rightarrow$) evaluation and coevaluation maps.
They are related by 
$\evR_V = \evL_{\dualL{V}} \circ \left( \delta_V \otimes \id_{\dualL{V}} \right)$, etc.

Given a morphism $f : A \otimes C \to B \otimes C$ in $\cat$, its \textsl{right partial trace} over $C$ is defined to be (we omit writing `$\otimes$', 
	say $1$ instead of $\id$, and write $\xrightarrow{\sim}$ for the coherence isomorphisms)
\begin{align}
	\mathsf{tr}^r_C(f)
	= \big[ &
	A \xrightarrow{\sim} A \one \xrightarrow{1\,\coevL_{\!\!C}} A (C  C^\vee)
	 \xrightarrow{\sim} 
	(A C) C^\vee
	\xrightarrow{f \, 1} (B  C) C^\vee
	\nonumber \\ &
	\xrightarrow{\sim} B(CC^\vee) \xrightarrow{ 1\,\evR_{\!\!C} } B\one 
	\xrightarrow{\sim} B  
	\big] \ .
\end{align}
Analogously, the \textsl{left partial trace} over $C$ of a morphism $g:  C \otimes A \to C \otimes B$ is
\begin{align}
	\mathsf{tr}^l_C(g)
	= \big[ &
	A \xrightarrow{\sim} \one A  \xrightarrow{\coevR_{\!\!C}\,1} (C^\vee  C) A
	 \xrightarrow{\sim} 
	C^\vee  (C A)
	\xrightarrow{1 \, g } 	C^\vee  (C B)
	\nonumber \\ &
	\xrightarrow{\sim} 	(C^\vee  C)B
	 \xrightarrow{ \evL_{\!\!C} \, 1} \one B 
	\xrightarrow{\sim} B  
	\big] \ .
\end{align}

The original definition of modified traces from \cite{GPMV,GKPM1} is for general tensor ideals, but here we will restrict our attention to the tensor ideal
\begin{align}
		\proj ~\subset~ \cat \ ,
\end{align}
the full subcategory
consisting of all projective objects in $\cat$.
We note that unless $\cat$ is semisimple, the categorical trace vanishes identically on $\proj$, 
see e.g.\ \cite[Rem.\,4.6]{GR}.

\medskip

\begin{mydef}\label{def:mod-tr}~
\begin{enumerate}\setlength{\leftskip}{-1.7em}
\item[\em i)]
    A \textsl{right (left) modified trace} on $\proj$
    is a family of linear functions
    \begin{align}\label{modTrFamily}
        \{ \modTr_P : \End_{\cat}(P) \to \field \}_{P\in \proj}
    \end{align}
    satisfying two conditions, \textsl{cyclicity} and \textsl{right (left) partial trace property}, 
    given as follows.
    \begin{itemize}
        \item[1.] (\textsc{Cyclicity}) If $P,P'\in \proj$, then for all $f:P\to P'$, $g:P'\to P$ 
            we have
            \begin{align}\label{def_cyclicity}
                \modTr_P(g\circ f) = \modTr_{P'}(f\circ g)\ .
            \end{align}

        \item[2.] (\textsc{Right Partial Trace Property}) 
            If $P\in \proj$ and $V\in \cat$, then for all $f\in \End_{\cat}(P\otimes V)$
            \begin{align}\label{def_rightpartialtraceproperty}
                \modTr_{P\otimes V} (f) = \modTr_P\left( \mathsf{tr}^r_V(f) \right)\ .
            \end{align}

        \item[2$'$.] (\textsc{Left Partial Trace Property}) 
            If $P\in \proj$ and $V\in \cat$, then for all $f\in \End_{\cat}(V\otimes P)$
            \begin{align}\label{def_leftpartialtraceproperty}
                \modTr_{V\otimes P} (f) = \modTr_P\left( \mathsf{tr}^l_V(f) \right)\ .
            \end{align}
    \end{itemize}
    If such a family $\modTr_{\bullet}$ satisfies both the left and the right partial trace property,
    then it is simply called a \textsl{modified trace}.

\item[\em ii)]
    A right (resp.\ left)    
    modified trace $\modTr_{\bullet}$ is \textsl{non-degenerate} if the pairings
\begin{align}\label{def_modTr_pairing}
	\Hom_{\cat}(M,P) \times \Hom_{\cat}(P,M) \to \field,
	\qquad
	(f,g) \mapsto \modTr_P(f\circ g)\ ,
\end{align}
are non-degenerate for all $M\in \cat, P\in \proj$.
\end{enumerate}
\end{mydef}

In case $\field$ is algebraically closed and $\cat$ is unimodular (i.e.\ the socle and top of the projective cover of the tensor unit $\one$ are both the tensor unit), non-zero left and right modified traces as above exist and are unique up to scalars \cite[Sec.\,5.3]{GKPM2}.
Furthermore, these left/right modified traces are non-degenerate.
This significantly generalises earlier existence and uniqueness results, see 
e.g.\ \cite{GKPM1, GR, BBG}.

\medskip

We will focus on $\cat$ being $\hmod$, the category of finite-dimensional modules over a pivotal 
unimodular quasi-Hopf algebra.
In this situation we obtain existence, uniqueness and non-degeneracy of non-zero left/right modified traces without requiring $\field$ to be algebraically closed
(Theorem~\ref{thm:symcoint-is-modtr} below), 
and by using  methods different from those in~\cite{GKPM2}.

\section{Cointegrals for quasi-Hopf algebras}\label{Sec:qHopf}

In this section we recall some definitions and properties related to quasi-Hopf algebras that we shall need.
We start by giving our conventions for quasi-Hopf algebras, and then proceed to define integrals, cointegrals and symmetrised cointegrals for quasi-Hopf algebras (which have to be pivotal and unimodular in the latter case).

\medskip

Throughout this section, $H$ denotes a finite-dimensional quasi-Hopf algebra over $\field$.

\subsection*{Quasi-Hopf algebras}
    The antipode of $H$ is denoted by $S$ and the coassociator and its inverse by 
    $\coassQ$ and $\invCoassQ = \coassQ\inv$.
    The evaluation and coevaluation element are $\alphaQ,\betaQ \in H$, and without loss
    of generality we assume 
    \begin{align}
        \counit(\alphaQ) = 1 = \counit(\betaQ)\ .
    \end{align}
    Using sumless Sweedler notation for both (iterated) coproducts and elements in tensor
    powers of $H$, we write for example
    \begin{align}
        \Delta(h) &= h\sweedler{1} \otimes h\sweedler{2}
        , \quad
        (\Delta \otimes \id)(\Delta(h)) =
        h\sweedler{1,1}\otimes h\sweedler{1,2} \otimes h\sweedler{2} ,
        \notag \\
        \coassQ &= \coassQ_1 \otimes \coassQ_2 \otimes \coassQ_3 \in H^{\otimes 3}
    \end{align}

    Following the conventions from \cite[Sec.\,6]{FGR1} for the axioms,
    the comultiplication $\Delta:H\to H\otimes H$ satisfies
    \begin{align}
        \left( \Delta \otimes \id \right) (\Delta(h)) 
        \cdot \coassQ
        = 
        \coassQ \cdot 
        \left( \id \otimes \Delta \right) (\Delta(h)) \ ,
    \end{align}
    or in index notation
    \begin{align}
        h\sweedler{1,1} \coassQ_1 \otimes 
        h\sweedler{1,2}\coassQ_2  \otimes 
        h\sweedler{2} \coassQ_3 
        =
        \coassQ_1 h\sweedler{1}\otimes 
        \coassQ_2 h\sweedler{2,1} \otimes 
        \coassQ_3 h\sweedler{2,2}
    \end{align}
    for all $h\in H$.
    Note that this is the opposite of the convention that, for example,
    \cite{HN-integrals,BC2} are using.
    
The category $\cat \vcentcolon= \hmod$ of finite-dimensional $H$-modules 
is a monoidal category with associator
    \begin{align}
        \coassQ_{U,V,W}:U\otimes (V\otimes W) &\to (U\otimes V)\otimes W 
        \notag \\ \label{coassociator_categorical}  
        u\otimes v\otimes w &\mapsto \coassQ_1 u \otimes \coassQ_2 v \otimes \coassQ_3 w\ .
    \end{align}
    
If $x\in H\otimes H$, then we will frequently write 
\begin{align}
x_{21} \vcentcolon= \tau(x) = x_2\otimes x_1\in H\otimes H\ ,
\end{align}
where $\tau$ is the
tensor flip
of vector spaces. 
This notation is extended to higher tensor powers in the obvious way.

Following \cite{HN-integrals}, for $h,a\in H$, $f\in H^*$ we write
\begin{align}
	(h \rightharpoonup f)(a) = f ( a h )
	\quad , \qquad 
	&&
	f \rightharpoonup h = h\sweedler{1} f(h\sweedler{2}) \ ,
	\nonumber\\
	(f \leftharpoonup h)(a) = f ( h a )
	\quad , \qquad 
	&&
	h \leftharpoonup f = f(h\sweedler{1}) h\sweedler{2} \ .
\label{eq:hookActionDefined}
\end{align}

    One can also consider the \textsl{opposite} and the
    \textsl{coopposite} quasi-Hopf algebras $H\op$ (with opposite multiplication) 
    and $H\cop$ (with opposite comultiplication).
    These algebras become quasi-Hopf algebras after modifying the defining data
    according to $S\op = S\cop = S\inv$, $\coassQ\op = \invCoassQ$,
    $\alphaQ\op  = S\inv(\betaQ)$, $\betaQ\op = S\inv(\alphaQ)$,
    $\coassQ\cop = \invCoassQ_{321} = 
    \invCoassQ_3\otimes \invCoassQ_2\otimes \invCoassQ_1$,
    $\alphaQ\cop = S\inv(\alphaQ)$ and $\betaQ\cop = S\inv(\betaQ)$.

\subsection*{Dual modules}

    For $V\in\cat$ the left dual $\dualL{V}$ is the dual vector space $V^*$, with action
    \begin{align}
        \langle h.v^*, w \rangle = \langle v^*, S(h) w \rangle
        \quad \text{for all }h\in H, v^*\in V^*, w\in V,
    \end{align}
    and the (left) rigid structure on $\hmod$ is then defined as follows.
    The left evaluation are 
    \begin{alignat}{2}
        \evL_V : \dualL{V} \otimes V &\to \one \ , 
        \quad 
        \evL_V(v^*\otimes v) 
        &&= \langle v^*,~ \alphaQ v\rangle\ ,
    \end{alignat}
    respectively. 
    Using a basis $\{v_i\}$ of $V$ with corresponding dual basis $\{v^i\}$ we can
    write the left 
    coevaluation as
    \begin{alignat}{2}
        \coevL_V : \one &\to V \otimes \dualL{V} \ , 
        \quad 
        \coevL_V(1) 
        &&= \sum_i \betaQ v_i \otimes v^i
    \end{alignat}
    One can define right duals analogously in terms of $S^{-1}$, but we will not do this
    here as below we will work with pivotal quasi-Hopf algebras, where we will use the
    pivotal structure to define right duals.     
       
    \medskip

We extend the hook notation from \eqref{eq:hookActionDefined} to dual vector spaces,
so that for example the action on the left dual of the $H$-module $V$ could then be
written as $v^*\leftharpoonup S(h)$, since
\begin{align}
    \langle h.v^* \mid w \rangle
    =
    \langle v^* \mid S(h) w \rangle
    =
    \langle v^* \leftharpoonup S(h) \mid w \rangle
\end{align}
for all $h\in H$, $w\in V$, $v^*\in V^*$.

    The regular left and right action by an element $h\in H$ is denoted
    by $l_h$ and $r_h$, respectively, so that for all $a\in H$
    \begin{align}\label{eq:left-right-mult}
        l_h(a) = ha \qquad \text{and} \qquad r_h(a) = ah\ .
    \end{align}

    \medskip

\begin{figure}[tb]
    \begin{align*}
		\scalebox{0.7}{
			\ipic{-0.5}{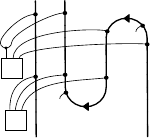}{4.3}
			\put(-172,-70){\scalebox{1.3}{$\coassQ$}}
			\put(-177,-5) {\scalebox{1.3}{$\invCoassQ$}}
			\put(-120,-52) {\scalebox{1.3}{$\betaQ$}}
			\put(-25,35) {\scalebox{1.3}{$\alphaQ$}}
			\put(-147,-098){$H$}
			\put(-147,+090){$H$}
			\put(-110,+090){$H$}
			\put(-008,-098){$H$}
		}
		\ = \
		\scalebox{0.7}{
			\ipic{-0.5}{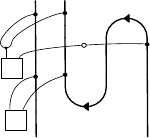}{4.3}
			\put(-174,-69){\scalebox{1.3}{$\pR$}}
			\put(-179,-5) {\scalebox{1.3}{$\qR$}}
			\put(-147,-098){$H$}
			\put(-147,+090){$H$}
			\put(-110,+090){$H$}
			\put(-008,-098){$H$}
		}
		\ = \
		\id_{H\otimes H}
    \end{align*}
	\caption{The equalities in \eqref{eq:identities qpL} follow from the zig-zag identities for duals in $\hmod$, where both sides are tensored with the identity. 
		We show this above for first equality in \eqref{eq:identities qpL}, where we have written out $\id_H$ times the zig-zag identity for $H$ in $\hmod$ as a string diagram in $\Vect$.}
	\label{fig_qRpR_relation}
\end{figure}
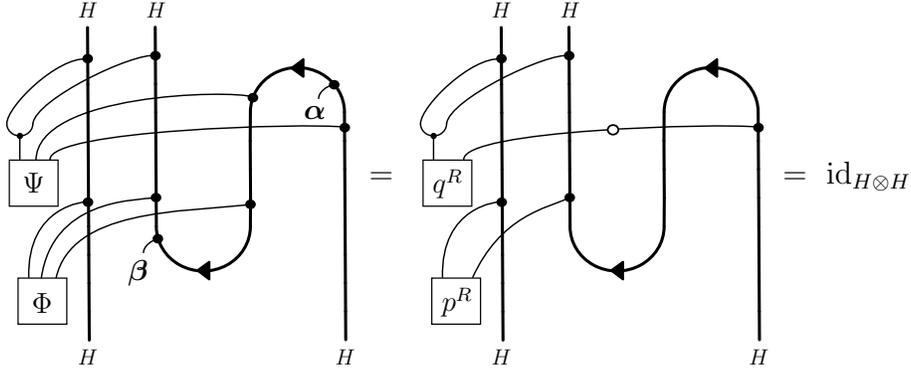

We will make use of the special elements $\qR,\pR,\qL,\pL \in H\otimes H$,
which have already appeared in \cite{Dr}, and have been further 
used in \cite{HN-doubles} and subsequent papers concerning quasi-Hopf
algebras.
They are defined as
\begin{alignat}{2}
	\notag
	&\qR\vcentcolon= \invCoassQ_1\otimes S\inv(\alphaQ \invCoassQ_3)\invCoassQ_2 \ ,
	\qquad 
	&& \pR \vcentcolon= \coassQ_1\otimes \coassQ_2 \betaQ S(\coassQ_3) \ , \\
	\label{eq:q,\pL}
	&\qL\vcentcolon= S(\coassQ_1) \alphaQ \coassQ_2 \otimes \coassQ_3 \ , \qquad
	&& \pL \vcentcolon= \invCoassQ_2 S\inv(\invCoassQ_1 \betaQ) \otimes \invCoassQ_3 \ ,
\end{alignat}
and satisfy the identities
\begin{alignat}{2}
	\notag
	\Delta(\qR_1) \pR [\oneQ\otimes S(\qR_2)] &= \oneQ\otimes \oneQ \ , \qquad 
	& [\oneQ\otimes S\inv(\pR_2)] \qR \Delta(\pR_1) &=\oneQ\otimes \oneQ \ , \\
	\label{eq:identities qpL}
	\Delta(\qL_2) \pL [S\inv(\qL_1)\otimes \oneQ] &= \oneQ\otimes \oneQ \ , \qquad 
	& [S(\pL_1)\otimes \oneQ] \qL \Delta(\pL_2) &=\oneQ\otimes \oneQ\ 
\end{alignat}
and
\begin{align}
    (\oneQ \otimes S\inv(a\sweedler{2}))
    \qR \Delta(a\sweedler{1}) 
    &=
    (a\otimes \oneQ)  \qR \ ,
    \notag \\ 
    \Delta(a\sweedler{1})\pR  (\one \otimes S(a\sweedler{2}))
    &= \pR  (a\otimes \oneQ)\ .
    \label{rel_pr_coproduct}
\end{align}
These identities are most easily understood using the standard graphical calculus\footnote{
	Our string diagrams are read from bottom to top. The further conventions we use for the graphical notation are detailed e.g.\ in \cite[Sec.\,2.1,\,2.2,\,6.1]{FGR1}.\label{fn:stringconvention}}
 for
    rigid monoidal categories, see for example in Figure~\ref{fig_qRpR_relation}.
    More about this may be found in \cite[Sec.\,2]{HN-doubles}.

\subsection*{Pivotal quasi-Hopf algebras}

To state the definition of a pivotal element, we need to recall the \textsl{Drinfeld twist}, which is an invertible 
element $\Dt\in H\otimes H$ such that
\begin{align}\label{eq_drinfeld_twist_comult}
\Dt  \Delta(S(a))  \Dt\inv = (S\otimes S)\left( \Delta\cop(a) \right), 
\end{align}
for all $a\in H$.
It corresponds to the canonical natural isomorphism 
\begin{align}
\gamma_{W,V}: \dualL{W}\otimes \dualL{V} \xrightarrow{\sim} \dualL{(V\otimes W)}
\end{align}
of $H$-modules, via 
\begin{align}
\left( \gamma_{W,V}( \varphi\otimes \psi) \right) (v\otimes w)
= (\psi\otimes\varphi) (\Dt_1v \otimes \Dt_2w)\ ,
\end{align}
for all $v\in V, w\in W, \varphi\in W^*, \psi\in V^*$.
For more details see e.g.\ \cite[Sec.\,6.2]{FGR1}.
We will not
need the explicit description of $\Dt$ here, but let us remark that in 
our conventions we have 
\begin{align}
(\counit \otimes \id) (\Dt) = \oneQ = (\id \otimes \counit)(\Dt)\ .
\end{align}
We can now state
(see \cite[Def.\,4.1]{BBG} for the Hopf case and \cite[Def.\,3.2]{BT2} for the quasi-Hopf case, where the name ``sovereign'' is used instead).

\begin{mydef}
$H$ is called \textsl{pivotal} 
	if there is an invertible element
	$\pivotQ \in H$, called the \textsl{pivot}, satisfying
	\begin{align}\label{def_coproduct_pivot}
	\Delta(\pivotQ) = \Dt\inv \cdot (S\otimes S)(\Dt_{21}) \cdot (\pivotQ \otimes \pivotQ)   \ ,
	\end{align}
	and such that $S^2(h)=\pivotQ h \pivotQ^{-1}$, for all $h\in H$. 
\end{mydef}

The pivot 
satisfies
\cite[Prop.\,3.12]{BT2}
\begin{align}
\counit(\pivotQ) = 1 
\qquad \text{and} \qquad
S(\pivotQ) = \pivotQ\inv\ .
\end{align}
We remark that the second property stems from the more general fact that in every pivotal category we have the identity $\delta_{\dualL{V}}\inv = \dualL{\left(\delta_V\right)}$,
see e.g.\ \cite[Lem.\,4.11]{Selinger}.

\medskip

\begin{rmrk}~ 
	\begin{enumerate}\setlength{\leftskip}{-1em}
		\item
		A pivot is not necessarily unique. 
		For example, if 
$z$ is central,
		invertible and satisfies $\Delta(z)=z\otimes z$,
		then $z\pivotQ$ is also a pivot.
		We will indicate our choice of pivot by saying that
		$(H,\pivotQ)$ is a pivotal quasi-Hopf algebra.

		\item
		If $(H,\pivotQ)$ is pivotal,
		then indeed $\hmod$ is pivotal. 
		The pivotal structure is the monoidal natural isomorphism with components
		\begin{align}\label{pivotal_structure_hmod}
		\delta_V : V \to \ddualL{V},
		\qquad
		\delta_V = \delta^{\Vect}_V \circ l_{\pivotQ}\ ,\qquad 
		V\in \hmod\ ,
		\end{align}
		where $\delta^{\Vect}$ is the canonical pivotal structure of 
		underlying vector spaces.
		In fact, the set of pivotal structures on $\hmod$ is in bijective correspondence 
		with the set of pivots for $H$.
		For a proof see e.g.\ \cite[Prop.\,3.2]{BCT}, and note 
		that our pivot is their inverse pivot.
	\end{enumerate}
\end{rmrk}

\label{right_rigid_structure}
Using the pivotal structure we define \textsl{right duals} in the standard way. 
Namely, as an object the right dual $\dualR{V}$ of an object $V$ is just the left dual
$\dualL{V}$.
The right evaluation and coevaluation are then defined as
\begin{align}
\evR_V \vcentcolon= 
\evL_{\dualL{V}} \circ \left( \delta_V \otimes \id_{\dualL{V}} \right) 
\ , \qquad
\coevR_V \vcentcolon= 
\left( \id_{\dualL{V}} \otimes \delta_V\inv \right) \circ \coevL_{\dualL{V}}\ .
\end{align}
Explicitly, for example,
\begin{align}\label{eq:right-eval-explicit}
\evR_V(v \otimes w^*) &= 
\big\langle  \, \pivotQ v \, , \,  
	w^* \leftharpoonup S(\alphaQ) \, \big\rangle
\notag \\
	&= \big\langle  \, S(\alphaQ)\pivotQ v \, , \,  w^*  \, \big\rangle
	= \big\langle  \, \pivotQ S^{-1}(\alphaQ) v \, , \,  w^*  \, \big\rangle \ ,
\end{align}
where in the last step we used $S(h) \pivotQ = \pivotQ S^{-1}(h)$ for $h \in H$.

\subsection*{Integrals and cointegrals}

A \textsl{left integral}\footnote{
	In \cite{BBG} this is called a left cointegral, and what we call cointegral here is called integral there (all in the case of Hopf algebras).
We follow the conventions in e.g.\ \cite{HN-integrals,BC1,BC2}, from where we take the definition of cointegrals for quasi-Hopf algebras. 
}
for $H$ is an element $\Lambda$ of $H$, such that 
$h\Lambda = \counit(h)\Lambda$ for all $h\in H$.
One similarly defines a \textsl{right integral} for $H$ to be a left integral for $H\op$.
The spaces of left and right integrals are always one-dimensional for 
a finite-dimensional quasi-Hopf algebra, see for example \cite[Sec.\,2]{BC1}.

The difference between left and right integrals is measured by the \textsl{modulus} of $H$.
This is the unique algebra morphism 
$\modulus:H\to \field$ such that for any left (resp.\ right) integral $\Lambda$ (resp.\ $\Lambda'$)
we have $\Lambda h = \modulus(h) \Lambda$ (resp.\ $h \Lambda' = \modulus\inv(h)\Lambda'$).
Left and right integrals coincide if and only if the modulus of $H$ is given by the counit
$\counit$, and in that case we say that $H$ is \textsl{unimodular}.

\begin{rmrk}
	The category $\hmod$ is unimodular
	if and only if $H$ is unimodular, see \cite[Rem.\,3.2]{ENO}.
\end{rmrk}

There is also the `dual' notion of \textsl{cointegrals} for quasi-Hopf algebras, proposed 
in \cite{HN-integrals} and further studied in 
\cite{BC1,BC2}. 
To state their definition we need the elements $\elU,\elV\in H\otimes H$, given by
\begin{align}
\elU &\vcentcolon= \Dt\inv  (S\otimes S)(\qR_{21})\ ,
\notag \\
\label{eq_HN_V}
\elV &\vcentcolon= (S\inv \otimes S\inv)\left( \Dt_{21}  \pR_{21}\right)\ .
\end{align}

We also set
\begin{align}
    \label{eq:elsmallu-defined}
    \elu = (\modulus \otimes S^2)(\elV).
\end{align}

\begin{rmrk}
	It is easy to see that 
	\begin{align}
	(\qR)\cop = \qL_{21},~~
	(\pR)\cop = \pL_{21},~~
	(\qL)\cop = \qR_{21},~~\text{and}~~
	(\pL)\cop = \pR_{21}\ .
	\end{align}
	Likewise one finds $\Dt\cop = (S\inv \otimes S\inv)(\Dt)$, cf.\ 
	\cite[Sec.\,3]{BC2}, and therefore
	\begin{align}\notag
        \elU\cop &\vcentcolon= (S\inv\otimes S\inv)(\qL \Dt\inv)\ ,\\
        \elV\cop &\vcentcolon= (S \otimes S)\left( \pL \right) \Dt_{21}\ ,\notag \\
        \elu\cop & = (\modulus \otimes S^{-2})(\elV\cop)
	\end{align}
	If $H$ is pivotal, then $\pivotQ\cop = \pivotQ\inv$.
\end{rmrk}

Paraphrasing \cite[Eq.\ (3.6) and Def.\,3.4]{BC2} we now define: 

\begin{mydef}\label{def:cointegrals}
	Let $\modulus$ be the modulus of $H$.
    \begin{itemize}
        \item
        A \textsl{left cointegral} for $H$ is an element $\cointL\in H^*$ satisfying
        \begin{align}\label{def_leftCointegral}
            (\id \otimes \cointL) (\elV \Delta(h)  \elU) =
            \modulus(\coassQ_1) \cointL(h S(\coassQ_2)) \coassQ_3
        \end{align}
        for all $h\in H$.
    \item
        A \textsl{right cointegral} for $H$ is a left cointegral for $H\cop$.
        In full detail  this means that $\cointR \in H^*$ is a right cointegral for $H$
        if and only if
        \begin{align}\label{def_rightCointegral}
            (\id \otimes \cointR) \big(\elV\cop \Delta\cop(h)  \elU\cop \big) =
            \modulus(\invCoassQ_3) \cointR(h S\inv(\invCoassQ_2)) \invCoassQ_1\ ,
        \end{align}
        for all $h\in H$.
    \end{itemize}
\end{mydef}

    Cointegrals for quasi-Hopf algebras have a categorical interpretation, see
    \cite{SS,BGR2}.
    Cointegrals satisfy a number of properties, and for convenience we collect the ones we
    will need in the following proposition.

    \begin{prop}\label{prop:coint-properties}
	Let $\modulus$ be the modulus of $H$.
        \begin{enumerate}
            \item 
                Left (resp.\ right) cointegrals exist and are unique up to scalar.
            \item 
                Non-zero left (resp.\ right) cointegrals are non-degenerate forms on $H$.
            \item 
                Let $\cointL$ be a left cointegral.
                Then, for all $a,b\in H$
                \begin{align}\label{eq:LNakayamageneral}
                    \cointL(S\inv(a)b) 
                    = 
                    \cointL(bS(a \leftharpoonup \modulus))
                \end{align}
            \item 
                Let $\cointR$ be a right cointegral. Then, for all $a,b\in H$
                \begin{align}\label{eq:RNakayamageneral}
                    \cointR(S(a)b) 
                    = 
                    \cointR\big(
                    b S\inv (\modulus \rightharpoonup a)
                    \big) \ .
                \end{align}
            \item 
                Let $\cointR$ be a right cointegral.
                Then
                \begin{align}\label{eq:leftCointFromRight}
                    \coint = (\cointR \leftharpoonup \elu) \circ S
                \end{align}
                is a left cointegral.
                \item
                    Let $\cointL$ be a left cointegral.
                    Then
                    \begin{align}\label{eq:rightCointFromLeft}
                        \coint = (\cointL \leftharpoonup \elu\cop) \circ S\inv
                    \end{align}
                    is a right cointegral.
        \end{enumerate}
    \end{prop}
    \begin{proof}
        The proofs of \textit{(1)}--\textit{(4)} can be found in
        \cite[Sec.\,4,\,5]{HN-integrals}.
        The last two points are \cite[Prop.~4.3]{BC2}.
    \end{proof}

The preceding proposition allows us to give the following equivalent characterisation of
cointegrals.\footnote{We thank the anonymous referee for explaining to us the improved result in Lemma~\ref{prop:NewVersion3.7}, which in the
originally submitted manuscript was only formulated in the unimodular case.}\textsuperscript{,}\footnote{
The shifted left cointegral $\cointSymL$  also appears in \cite[Sec.\,6.4]{SS} in
    relation to \emph{$\modulus$-twisted module traces}.
}
\begin{lemma}\label{prop:NewVersion3.7}
    Suppose that $H$ is pivotal, Let $\coint \in H^*$ and set 
    \begin{align}
        \cointSymR = \coint \leftharpoonup \elu \pivotQ
        \quad \text{and} \quad
        \cointSymL = \coint \leftharpoonup \elu\cop \pivotQ\inv .
    \end{align}
    Then
    \begin{enumerate}
        \item 
            $\coint$ is a left cointegral if and only if
            \begin{align}
        \label{eq:left-sym-coint-eq}
                (\id \otimes \cointSymL)
                \left( \qL \Delta(h) \pL \right)
                =
                \modulus(\invCoassQ_{3}) \cointSymL(\invCoassQ_2 h)\cdot 
                \pivotQ S\inv(\invCoassQ_1)
            \end{align}
        \item 
            $\coint$ is a right cointegral if and only if
            \begin{align}
                (\cointSymR \otimes \id)
                \left( \qR \Delta(h) \pR \right)
                =
                \modulus(\coassQ_{1}) \cointSymR(\coassQ_2 h)\cdot 
                \pivotQ\inv S(\coassQ_3)
            \end{align}
    \end{enumerate}
\end{lemma}

\begin{proof}
    We will prove the second part, the first statement is completely analogous.
    Let $\cointR\in H^*$ be a right cointegral.
    By Proposition \ref{prop:coint-properties} $(5)$, we thus have
    \begin{align}
        \cointSymR 
        = \cointR \leftharpoonup \elu \pivotQ
            = (\cointL \circ S\inv) \leftharpoonup \pivotQ
        = (\pivotQ\inv \rightharpoonup \cointL) \circ S\inv.
    \end{align}
    for a left cointegral $\cointL$. Using this equality and evaluating
 the left cointegral equation \eqref{def_leftCointegral} on
    $S\inv(h)\pivotQ\inv$ for $h\in H$ gives
    \begin{align}
        \label{eq:Proof_improved_3.7_first}
        (\id \otimes \cointL) (\elV \Delta(S\inv(h)\pivotQ\inv)  \elU) =
        \modulus(\coassQ_1) \cointSymR(\coassQ_2 h) \coassQ_3.
    \end{align}
    We have
    \begin{align}
        \Delta(\pivotQ\inv) \elU 
        &= (\pivotQ\inv \otimes \pivotQ\inv) (S\otimes S)(\qR_{21} \Dt_{21}\inv)
        \notag \\
        &= (S\inv\otimes S\inv)(\qR_{21} \Dt_{21}\inv) (\pivotQ\inv \otimes \pivotQ\inv).
    \end{align}
    Using $\elV=(S\inv\otimes S\inv)(\Dt_{21}\pR_{21})$ and
    \begin{align}
        \Delta(S\inv(h))
        &=
        (S\inv \otimes S\inv)(\Dt_{21} \Delta\cop(h) \Dt_{21}\inv)
    \end{align}
    we immediately simplify \eqref{eq:Proof_improved_3.7_first} to 
    \begin{align}
        \modulus(\coassQ_1) \cointSymR(\coassQ_2 h) \coassQ_3
        &=
        (\id \otimes \cointL)
        \left[
            (S\inv \otimes S\inv)
            \left( \qR_{21} \Delta\cop(h) \pR_{21} \right)
            \cdot (\pivotQ\inv \otimes \pivotQ\inv)
        \right]
        \notag \\
        &=
        ((r_{\pivotQ\inv} \circ S\inv) \otimes \cointSymR)
        \left( 
            \qR_{21} \Delta\cop(h) \pR_{21}
        \right)
        \notag \\
        &=
        ((S\inv \circ l_{\pivotQ\inv}) \otimes \cointSymR)
        \left( 
            \qR_{21} \Delta\cop(h) \pR_{21}
        \right)
    \end{align}
    Then, applying $S$ on both sides and multiplying with $\pivotQ\inv$ on the left gives
    \begin{align}
        (\cointSymR \otimes \id)
        \left( \qR \Delta(h) \pR \right)
        =
        \modulus(\coassQ_{1}) \cointSymR(\coassQ_2 h)\cdot 
        \pivotQ\inv S(\coassQ_3), 
    \end{align}
    as desired.
\end{proof}

We note that 
\begin{align}
    \cointSymL(ab) = \cointSymL((\modulus \rightharpoonup b) a)
    \quad \text{and} \quad
    \cointSymR(ab) = \cointSymR((b \leftharpoonup \modulus) a).
\end{align}
To see the first identity use Proposition~\ref{prop:coint-properties} (4) and (6) to
obtain
\begin{align}\label{eq:symmetricityPropsOfShiftedCoints}
    \cointSymL(ab) 
    &= \cointR(S(ab)\pivotQ) 
    = \cointR(S(a)\pivotQ S\inv(\modulus \rightharpoonup b))
    \notag \\
    &= \cointR( S(a)S(\modulus \rightharpoonup b) \pivotQ) 
    = \cointSymL( (\modulus \rightharpoonup b) a ).
\end{align}
The second identity can be seen using points (3) and (5) of the same proposition.

\medskip

\subsection*{Symmetrised cointegrals}
For the rest of this section we assume:
$$(H,\pivotQ) \ \text{is pivotal and unimodular} \ . $$

Note that in this case the elements $\elu$ and $\elu\cop$ from 
Proposition~\ref{prop:coint-properties} 
are both equal to~$\oneQ$.
The following immediate corollary to Proposition~\ref{prop:NewVersion3.7} will be
useful when comparing to modified traces.

\begin{cor}\label{coro_symmetrized_simple_condition}
    Let $\coint\in H^*$. 
    Then
    \begin{enumerate}
        \item 
            $\coint$ is a left cointegral if and only if
            \begin{align}\label{eq_symmetrized_left_cointegral_version_of_modTr}
                \cointSymL(h) ~ \oneQ = 
                \left( \pivotQ\inv \otimes \cointSymL\right) (\qL \Delta(h)
                \pL)
            \end{align}
            for all $h\in H$.
        \item 
            $\coint$ is a right cointegral if and only if
            \begin{align}\label{eq_symmetrized_right_cointegral_version_of_modTr}
                \cointSymR(h) ~ \oneQ = 
                \left( \cointSymR \otimes \pivotQ \right) (\qR \Delta(h) \pR)
            \end{align}
            for all $h\in H$.
    \end{enumerate}
\end{cor}

Following \cite[Sec.\,4]{BBG}, we call $\cointSymL$ and $\cointSymR$
the \textsl{symmetrised left} and \textsl{right cointegral}, respectively.
The adjective ``symmetrised'' is justified by the following corollary, which follows
from Proposition \ref{prop:coint-properties} (2), equation
\eqref{eq:symmetricityPropsOfShiftedCoints}, and the fact that here we assume $H$ to
be unimodular.

\begin{cor}[\hspace*{-.3em}{\cite[Prop.\,4.4]{BBG}}]
    \label{prop_symmetrized_is_nondegenerate}
    The non-zero symmetrised left (resp.\ right) cointegrals are non-degenerate
    symmetric linear forms on $H$.
\end{cor}

\begin{rmrk}\label{rem:new}
    Let $H$ be a finite-dimensional Hopf algebra.
    Following~\cite{Radford}, for any grouplike element $g\in H$, one can define the left
    ideal $L_g \subseteq H^*$ of \emph{left $g$-cointegrals} (called left $g$-integrals
    in~\cite{Radford}) as 
    \begin{align}
        L_g &= 
        \{ 
            \varphi \in H^* \mid 
            (\id \otimes \varphi) (\Delta(h)) = \varphi(h) g
            \quad
            \forall h\in H
        \}
    \end{align}
    These ideals are all one-dimensional \cite[Prop.~3]{Radford}.
    Indeed, note that $L_1$ is the space of left cointegrals.
    Then the linear isomorphism
    \begin{align}
        L_g \ni \varphi \mapsto (\varphi \leftharpoonup h) \in L_{h\inv g}
        \quad
        \text{for all grouplike } g, h,
    \end{align}
    shows $L_g \cong L_1$.
    Similarly one may define the space $R_g$ of \emph{right $g$-cointegrals}.
    Thus, if $(H,\pivotQ)$ is a pivotal Hopf algebra, Lemma~\ref{prop:NewVersion3.7}
    reduces to the statement
    \begin{align}
        (\cointL \leftharpoonup \pivotQ\inv) \in L_{\pivotQ}
        \quad \text{and} \quad
        (\cointR \leftharpoonup \pivotQ) \in R_{\pivotQ\inv}
        \ . 
    \end{align}
    For unimodular $H$, a symmetrized left cointegral is therefore a left $\pivotQ$-cointegral, and
    a symmetrized right cointegral is a right $\pivotQ\inv$-cointegral.
\end{rmrk}

\section{Modified traces for quasi-Hopf algebras}\label{sec:proof-modTr}

    Throughout this section $H$ will be a finite-dimensional quasi-Hopf algebra  
    over~$\field$.

    \subsection*{Tensoring with the regular representation}
    Let $V\in \hmod$.
    We denote by $\trivialMod{V}$ the vector space $V$ with trivial 
    $H$-module structure, i.e.\ $hv = \counit(h)v$ for $h\in H, v\in V$.
    Recall the definition of the elements $\pR$, etc., from \eqref{eq:q,\pL}.    
    We need the following generalisation of \cite[Thm.\,5.1]{BBG} to quasi-Hopf algebras 
    (see also \cite[Sec.\,2.3]{Sch}).

\begin{samepage}
    \begin{prop}\label{thm_tensor_powers_n_is_2}
~
        \begin{enumerate}
            \item The map
                \begin{align*}
                    \Fr:H \otimes \trivialMod{V} &\to H\otimes V,
                    \\
                    h\otimes v &\mapsto 
                    \left( \Delta(h) \pR \right) \cdot (1\otimes v) =
                    h\sweedler{1}\pR_1 \otimes h\sweedler{2}\pR_2 v
                \end{align*}
                is an isomorphism of $H$-modules, with inverse
                \begin{align*}
                    \Gr:H\otimes V & \to H\otimes \trivialMod{V},\\
                    h\otimes v &\mapsto 
                    \left[ (\id\otimes S)(\qR \Delta(h)) \right] \cdot (1\otimes v) \ .
                \end{align*}
            \item
                The map
                \begin{align*}
                    \Fl: \trivialMod{V} \otimes H &\to V\otimes H,
                    \quad
                    v\otimes h \mapsto 
                    \left( \Delta(h) \pL \right) \cdot (v\otimes 1)
                \end{align*}
                is an isomorphism of $H$-modules, with inverse
                \begin{align*}
                    \Gl:V \otimes H &\to \trivialMod{V} \otimes H,\\
                    v\otimes h &\mapsto 
                    \left[ (S\inv \otimes \id) (\qL \Delta(h)) \right] \cdot (v\otimes 1)\ .
                \end{align*}
        \end{enumerate}
    \end{prop}
\end{samepage}

\begin{proof}
    We only prove the first part, the second part is completely analogous. 
    It is obvious that $\Fr$ is an intertwiner, so we only need to show that $\Gr$ is
    a two-sided inverse.

    Recall the second identity in \eqref{rel_pr_coproduct},
    which can be graphically represented as
    \begin{equation}\label{eq:relations_HN}
        \scalebox{0.65}{
            \ipic{-0.5}{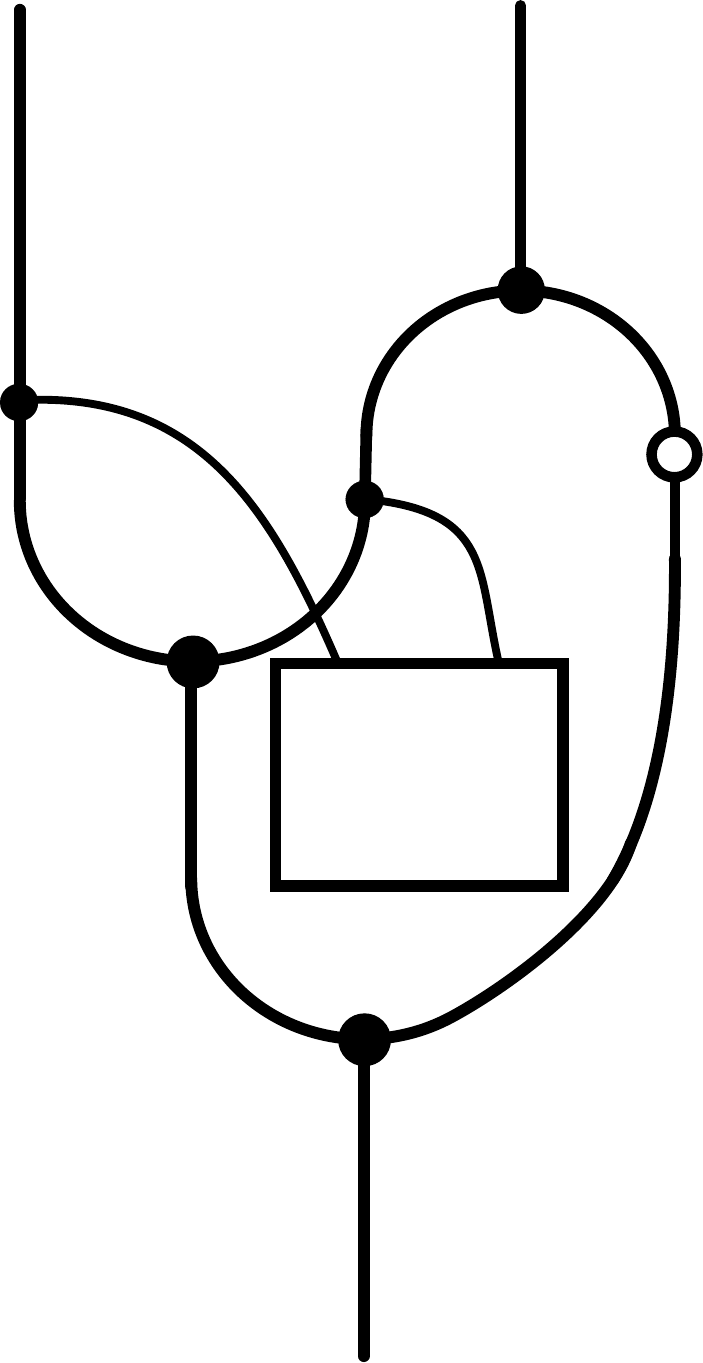}{0.4}
            \put(-39,-15){\scalebox{1.2}{$\pR$}}
        }
        \ =
        \scalebox{0.65}{
            \ipic{-0.5}{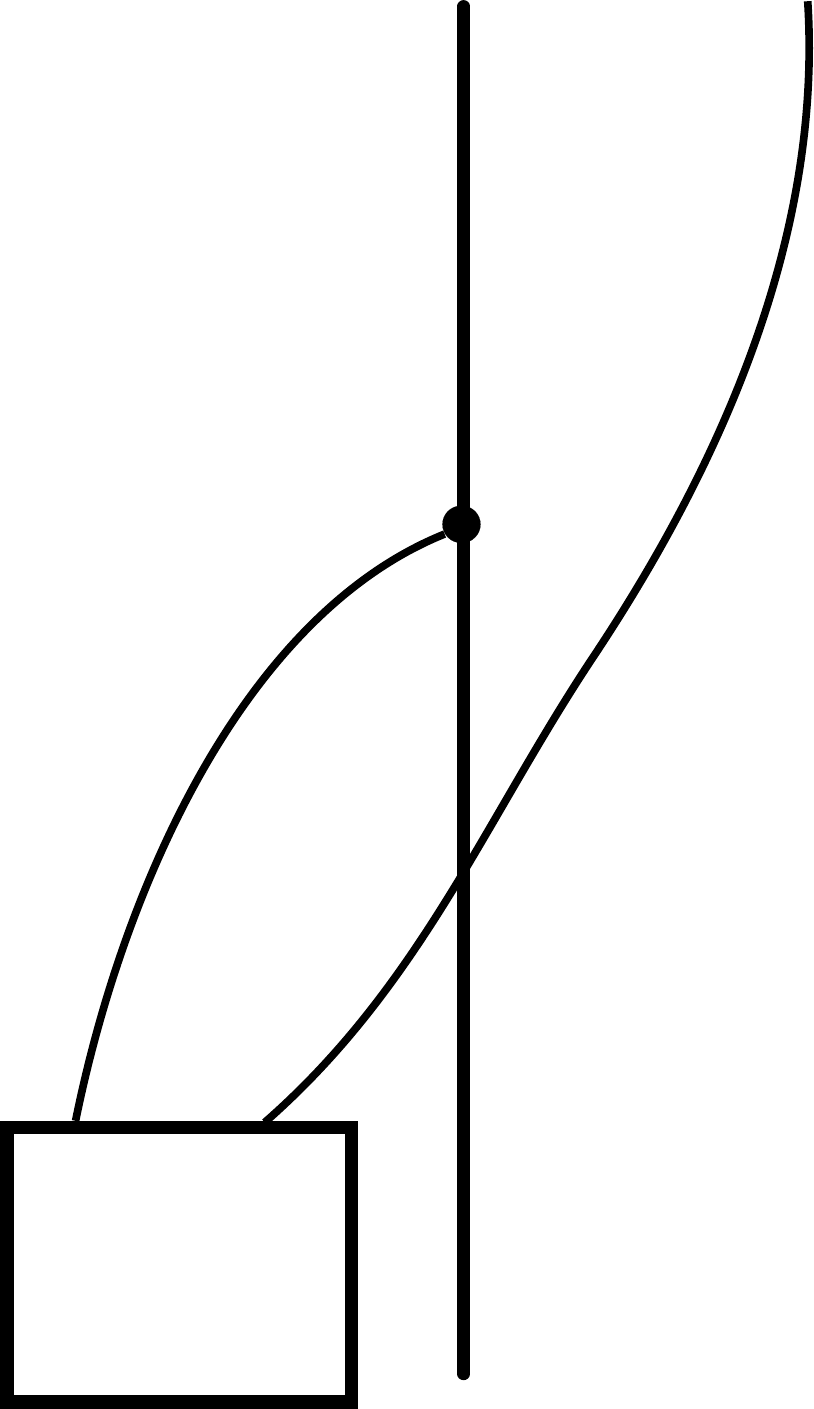}{0.4}
            \put(-80,-70){\scalebox{1.3}{$\pR$}}
        }
        \ .
    \end{equation}
    Using pictures we compute the composition
    $\Fr \circ \Gr$:
    \begin{equation}
        \scalebox{0.7}{
            \ipic{-0.5}{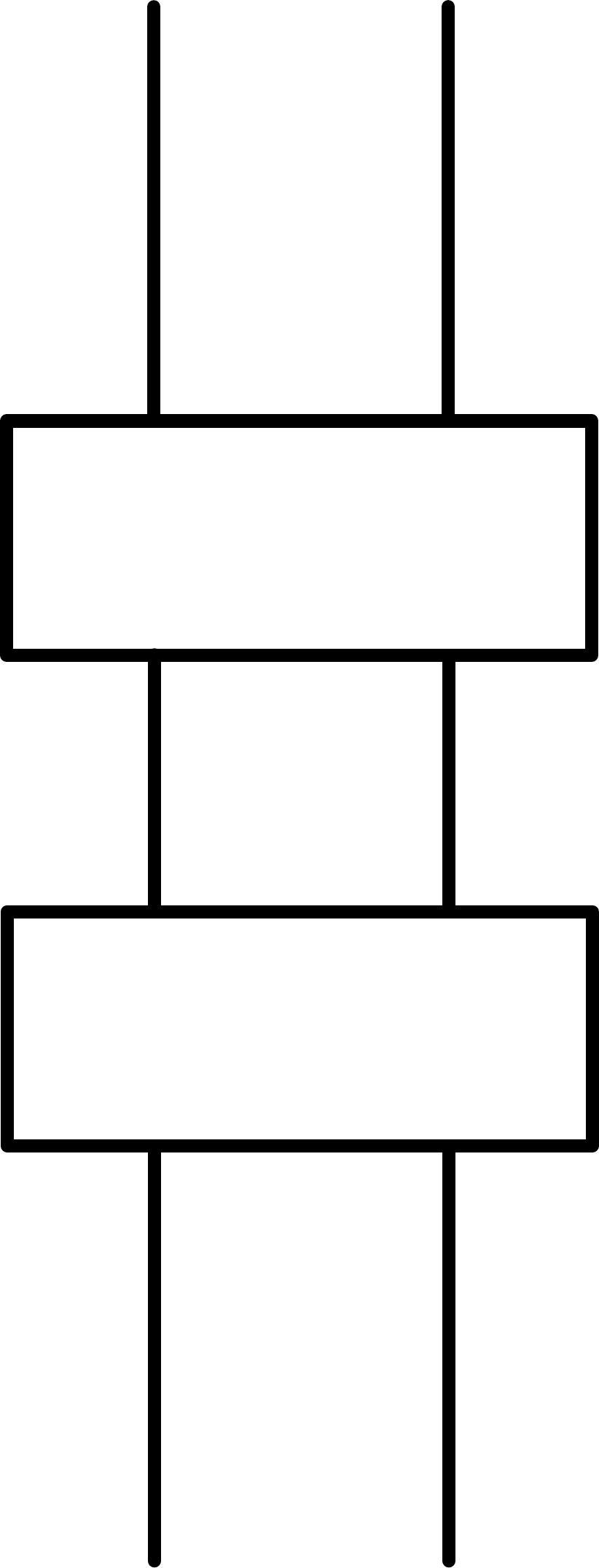}{0.23}
            \put(-33,-26){$\Gr$}
            \put(-33,18){$\Fr$}
        }
        \ =
        \scalebox{0.7}{
            \ipic{-0.5}{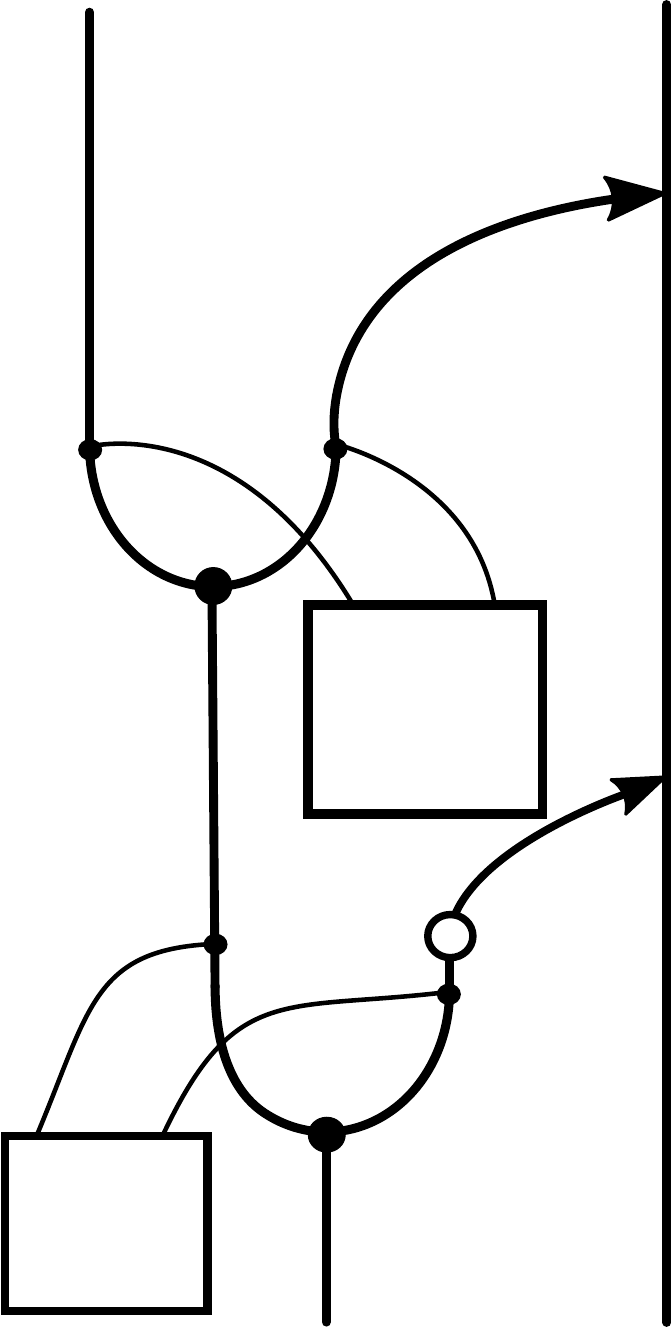}{0.4}
            \put(-72,-70){$\qR$}
            \put(-35,-10){$\pR$}
        }
        \ = 
        \scalebox{0.7}{
            \ipic{-0.5}{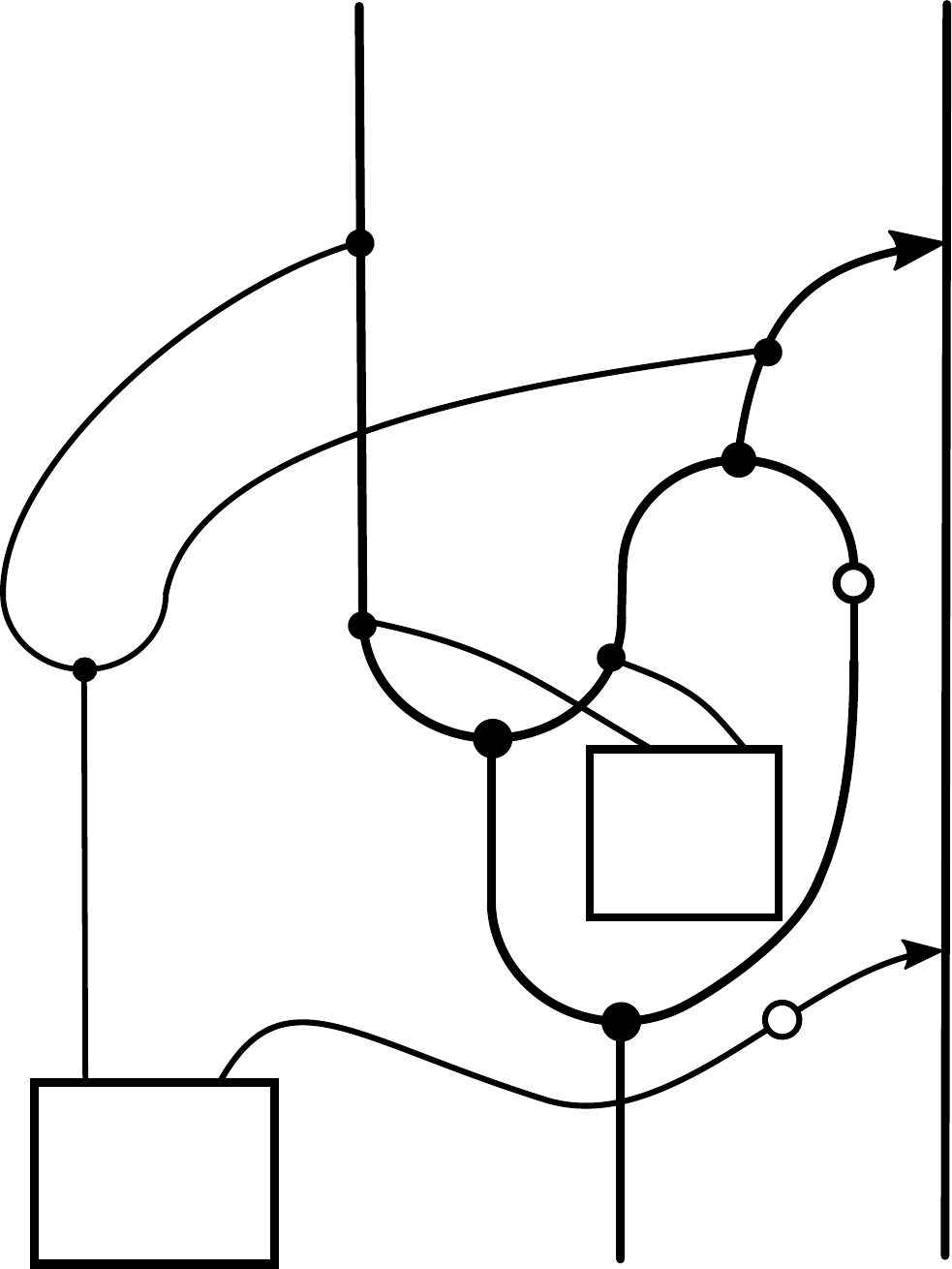}{0.4}
            \put(-102,-69 ){$\qR$}
            \put(-37,-28){$\pR$}
        }
        \stackrel{\scalebox{1.0}{\eqref{eq:relations_HN}}}{=} 
        \scalebox{0.7}{
            \ipic{-0.5}{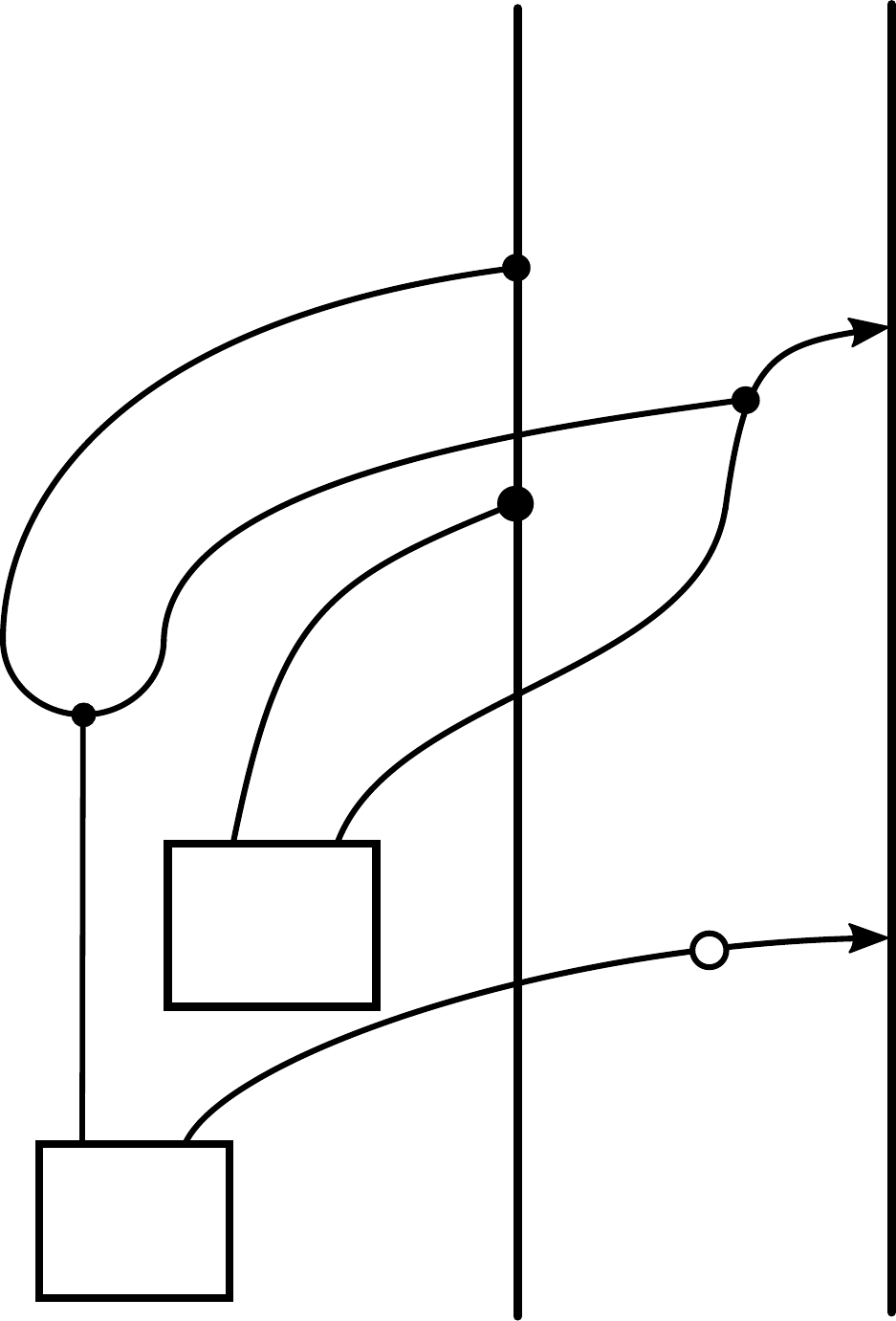}{0.4}
            \put(-98,-72){$\qR$}
            \put(-82,-35){$\pR$}
        }
        \ \stackrel{\scalebox{1.0}{\eqref{eq:identities qpL}}}{=}
        \scalebox{0.7}{
            \ipic{-0.5}{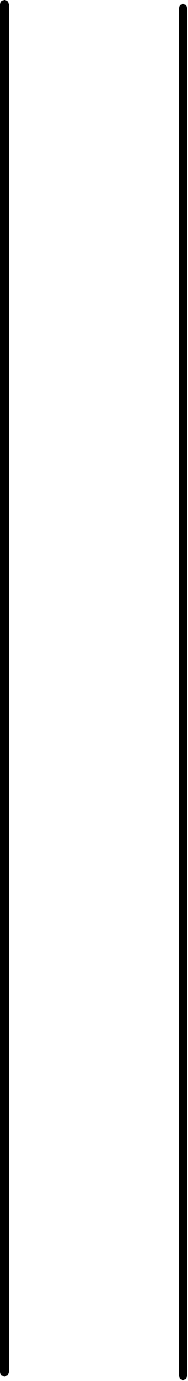}{0.4}
        }
    \end{equation}
    Similarly one shows $\Gr\circ \Fr = \id$.
    Since $\Fr$ is an intertwiner and bijective, 
    $\Gr$ necessarily is an intertwiner as well and we are done.
\end{proof}

As a consequence of the previous considerations we have the following lemma.

\begin{lemma}\label{lemma_endomorphisms_Hop_matrix}
Let $W$ be an $H$-module. 
The map
\begin{align}
\Xi:H\op \otimes_{\field} \End_\field(W) &\rightarrow \End_H(H\otimes W) \ ,
    \nonumber\\ 
a\otimes m&\mapsto \Fr \circ (r_a \otimes m) \circ \Gr
\end{align}
is an algebra isomorphism.
\end{lemma}

\begin{proof}
    Since for any
    algebra $A$ we have the algebra isomorphism 
    \begin{align}
        A\op \cong \End_A(A), \quad a \mapsto r_a \ ,
    \end{align}
    together with isomorphism property from Proposition~\ref{thm_tensor_powers_n_is_2}
    we see that the prescription
        $\Xi: (a\otimes m) \mapsto \Fr \circ (r_a\otimes m) \circ \Gr$
    is bijective, 
    and thus the isomorphism is established.
    It remains to be shown that the isomorphism is one of algebras.
    The multiplication for the endomorphism algebras is just composition, 
    and for $H\op \otimes_\field \End_\field(W)$ we simply take the one induced by 
    the tensor product of $\field$-algebras.
    Then the calculation
    \begin{align}
        \Xi(a \otimes m)\circ \Xi(b\otimes n)
        &= \Fr\circ (r_a\otimes m)\circ \Gr \circ \Fr (r_b\otimes n) \circ \Gr 
        \notag \\ \notag
        &= \Fr\circ \left( (r_a\circ r_b)\otimes (m\circ n) \right) \circ \Gr \\ 
        &= \Xi ((a\otimes m)\cdot (b\otimes n))
    \end{align}
    shows that $\Xi$ indeed preserves the algebra structure.               
\end{proof}

A similar result holds for $\End_H(W\otimes H)$.
    
    \subsection*{The main theorem}
    
    We will need the following extension result for symmetric linear forms:
    Let $A$ be a finite-dimensional unital $\field$-algebra. 
    By a \textsl{family of trace maps} $\{\modTr_P :\End_A(P) \to \field \}_{P\in \hpmod[A]}$
    (as opposed to left/right modified traces) we mean a family as in Definition~\ref{def:mod-tr}\,(i),
    which, however, only satisfies condition 1 (cyclicity) and not conditions 2 or 2' 
    (which do not make sense in $A$-mod). We have 
    (\!\!\cite[Prop.\,2.4]{BBG}, see also \cite[Prop.\,5.8]{GR}):

    \begin{prop}
        \label{prop_sym_lin_form_extends_uniquely}
        Let $A$ be a finite-dimensional unital $\field$-algebra. 
        Then a symmetric linear form $t$ on $A$ extends uniquely to a family of 
        trace maps 
        $\{\modTr_P :\End_A(P) \to \field \}_{P\in \hpmod[A]}$, given by
        \begin{align}
            \modTr_P(f) = \sum_{i=1}^n t( (b_i \circ f \circ a_i) (1) ),
            \quad f\in \End_A(P),
        \end{align}
        where $n$ depends on $P$, and $a_i:A \to P$, $b_i:P\to A$ satisfy
        \begin{align}
            \id_P = \sum_{i=1}^n a_i \circ b_i\ .
        \end{align}
        In particular 
        \begin{align}
            t_A(r_x) = t(x), x\in A\ .
        \end{align}
    \end{prop}

    The next lemma is an instance of the Reduction Lemma \cite[Lem.\,3.2]{BBG}
    when one takes $\cat = \hmod$ and $H$ as projective generator.

    \begin{lemma}\label{lem:redLemma}
        Let $H$ be pivotal with pivot $\pivotQ$.
        A symmetric linear function $t$ on $H$ extends to a 
        right modified trace on $\hpmod$ if and only if for all $f\in \End_H(H\otimes H)$
        \begin{align}\label{thm_redLemma_right}
            \modTr_{H\otimes H} (f) = \modTr_H\left( \mathsf{tr}^r_H(f)\right)
        \end{align}
        holds, where $\modTr_P$ is as in Proposition \ref{prop_sym_lin_form_extends_uniquely},
        for $P\in \hpmod$.

        Similarly, $\modTr$ extends to a left modified trace on $\proj[{}_H\cat[M]]$ if and only if
        \begin{align}\label{thm_redLemma_left}
            \modTr_{H\otimes H} (f) = \modTr_H\left( \mathsf{tr}^l_H(f)\right)
        \end{align}
        holds for all $f\in \End_H(H\otimes H)$.
    \end{lemma}

	We denote the subspace of symmetric forms $t\in H^*$ which extend to a right/left modified trace
	on $\hpmod$ by
	\begin{align}
	\symIntSpace^{r/l} \ .
	\end{align}
    Given $t\in \symIntSpace^{r/l}$, the corresponding modified trace $\modTr_\bullet$ takes the value
    \begin{align}\label{eq:symCoint-to-modTr}
        \End_H(H) \rightarrow \field
        \quad , \quad
        f \mapsto \modTr_H(f) = t(f(\one))
    \end{align}
	on the left regular module $H$.    

    We can now state the main theorem of our paper.
    Parts 2 and 3 generalise \cite[Thm.\,1]{BBG} to the setting of quasi-Hopf algebras. A stronger version of Part~1 was shown for Hopf algebras in \cite[Cor.\,6.1]{FOG}.

    \begin{thm}\label{thm:symcoint-is-modtr}
        Let $(H,\pivotQ)$ be a finite-dimensional pivotal quasi-Hopf algebra
        over~$\field$.
        We have:
        \begin{enumerate}
            \item
                A non-degenerate left (right) modified trace on $\hpmod$ exists if and only if $H$ 
                is unimodular.
		\end{enumerate}                
        Suppose now that $H$ is in addition unimodular. Then:
        \begin{enumerate} \setcounter{enumi}{1}
            \item 
                $\symIntSpace^{r/l}$ is equal to the space of symmetrised 
                right/left cointegrals.
                In particular, $\dim(\symIntSpace^{r/l})=1$.               
            \item 
			    A non-zero element of $\symIntSpace^{r/l}$ extends to a non-degenerate 
                right/left modified trace on $\hpmod$.
        \end{enumerate}
    \end{thm}
    \begin{proof}
        \textit{(1)}
        If $\modTr_\bullet$ is a non-degenerate left or right modified trace, then 
        $H\ni h\mapsto \modTr_H(r_h)\in \field$ is a non-degenerate symmetric linear form
        on $H$.
        Unimodularity follows from \cite[Prop.\,5.6]{HN-integrals}.
        The converse direction amounts to parts 2 and 3.

\smallskip

\noindent
        \textit{(2)}
        Suppose now that $H$ is unimodular, and let $\modTr_\bullet$ be a family of trace maps on $\hpmod$ (not necessarily left/right modified traces).
        Let $t\in H^*$ be the symmetric form on $H$ which corresponds to $\modTr_\bullet$ via Proposition~\ref{prop_sym_lin_form_extends_uniquely}.
        We will now compute both sides of \eqref{thm_redLemma_right} in Lemma~\ref{lem:redLemma} separately and then use that lemma to prove the statement.
        
Let $W \in \hmod$ and $f \in \End_H(H \otimes W)$.

\smallskip

\noindent
$\modTr_{H\otimes W}(f)$:
        By Lemma \ref{lemma_endomorphisms_Hop_matrix}, every 
        $f\in\End_H(H\otimes W)$ is of the form
        \begin{align}\label{eq:parametrise-End(HW)}
            f = \sum_{(a,m)} \Fr\circ (r_a \otimes m) \circ \Gr\ ,
        \end{align}
        where $a\otimes m$ is a simple tensor in $H\op \otimes \End_\field(W)$.
        For simplicity and without loss of generality we will assume that
        $f$ actually corresponds to the simple tensor $a\otimes m$.
        By cyclicity of $\modTr_\bullet$ we get
        \begin{align}\label{proof_rptc}
             \modTr_{H\otimes W}(f)
             &=
             \modTr_{H\otimes W}( \Fr \circ (r_a \otimes m) \circ \Gr )
             \notag \\
             &=
             \modTr_{H\otimes \trivialMod{W}}( r_a \otimes m )
             =
             \mathsf{tr}_\field(m) \ t(a)\ ,
        \end{align}
        where $\mathsf{tr}_\field(m)$ is the trace of the linear operator $m$.
        This can be seen by choosing any basis of $W$ and considering a decomposition 
        of $H\otimes \trivialMod{W}$ into $(\dim W)$ copies of $H$.
        
        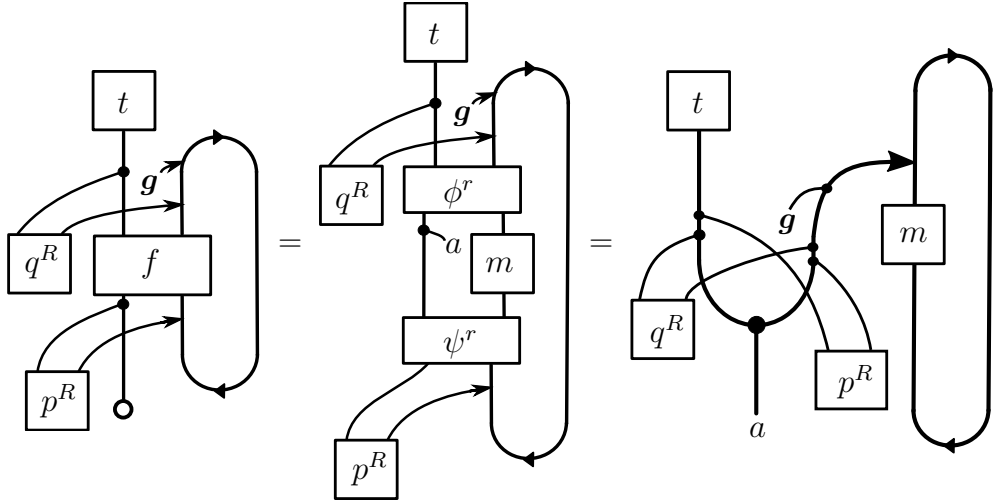
\begin{figure}[tb]
            \begin{equation*}
                \scalebox{0.8}{
                    \ipic{-0.5}{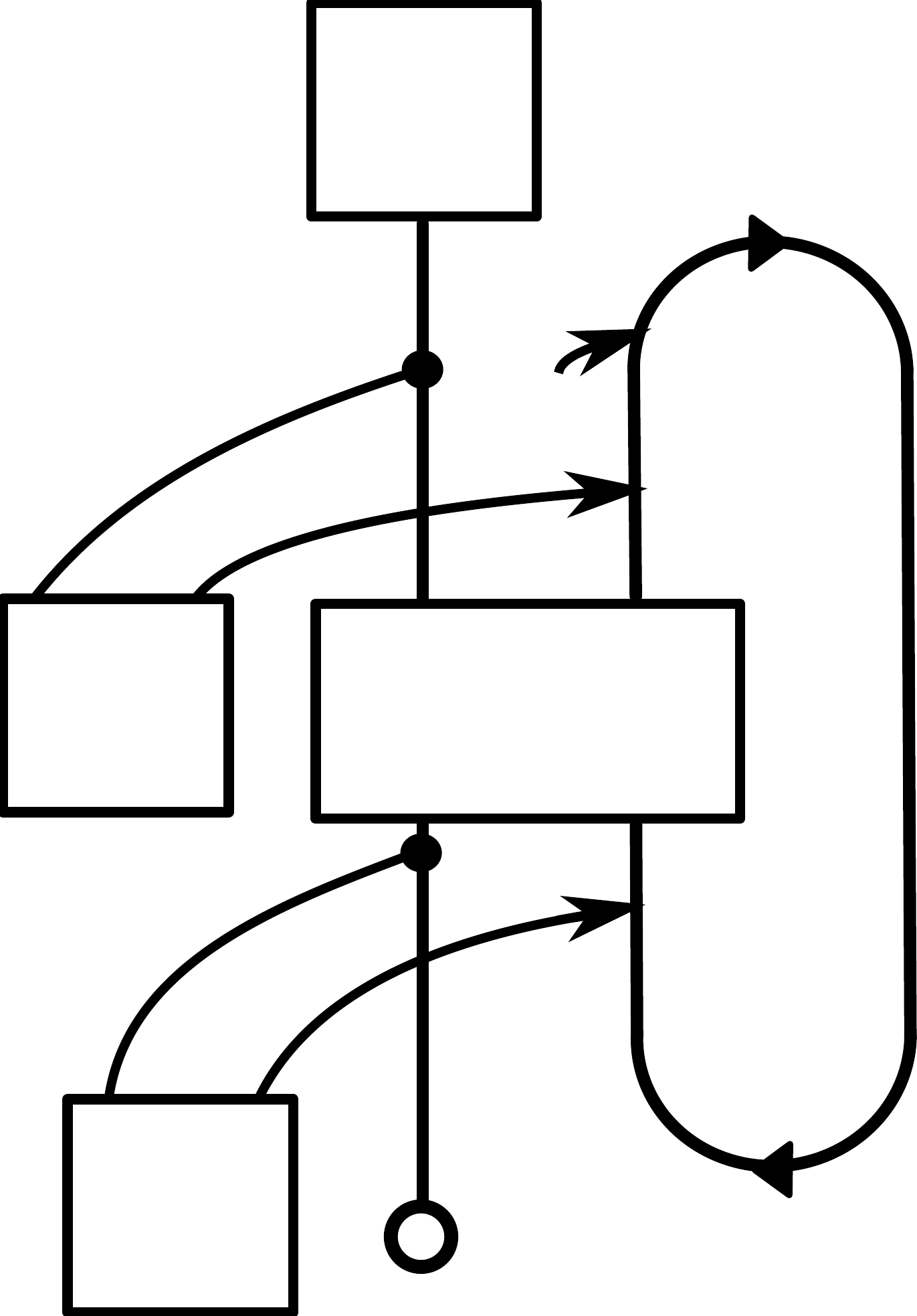}{0.3}
                    \put(-67,65){\scalebox{1.3}{$t$}}
                    \put(-56,29){\scalebox{1.2}{$\pivotQ$}}
                    \put(-56,-10){\scalebox{1.3}{$f$}}
                    \put(-111,-10){\scalebox{1.3}{$\qR$}}
                    \put(-103,-76){\scalebox{1.3}{$\pR$}}
                }
                \ =
                \scalebox{0.8}{
                    \ipic{-0.5}{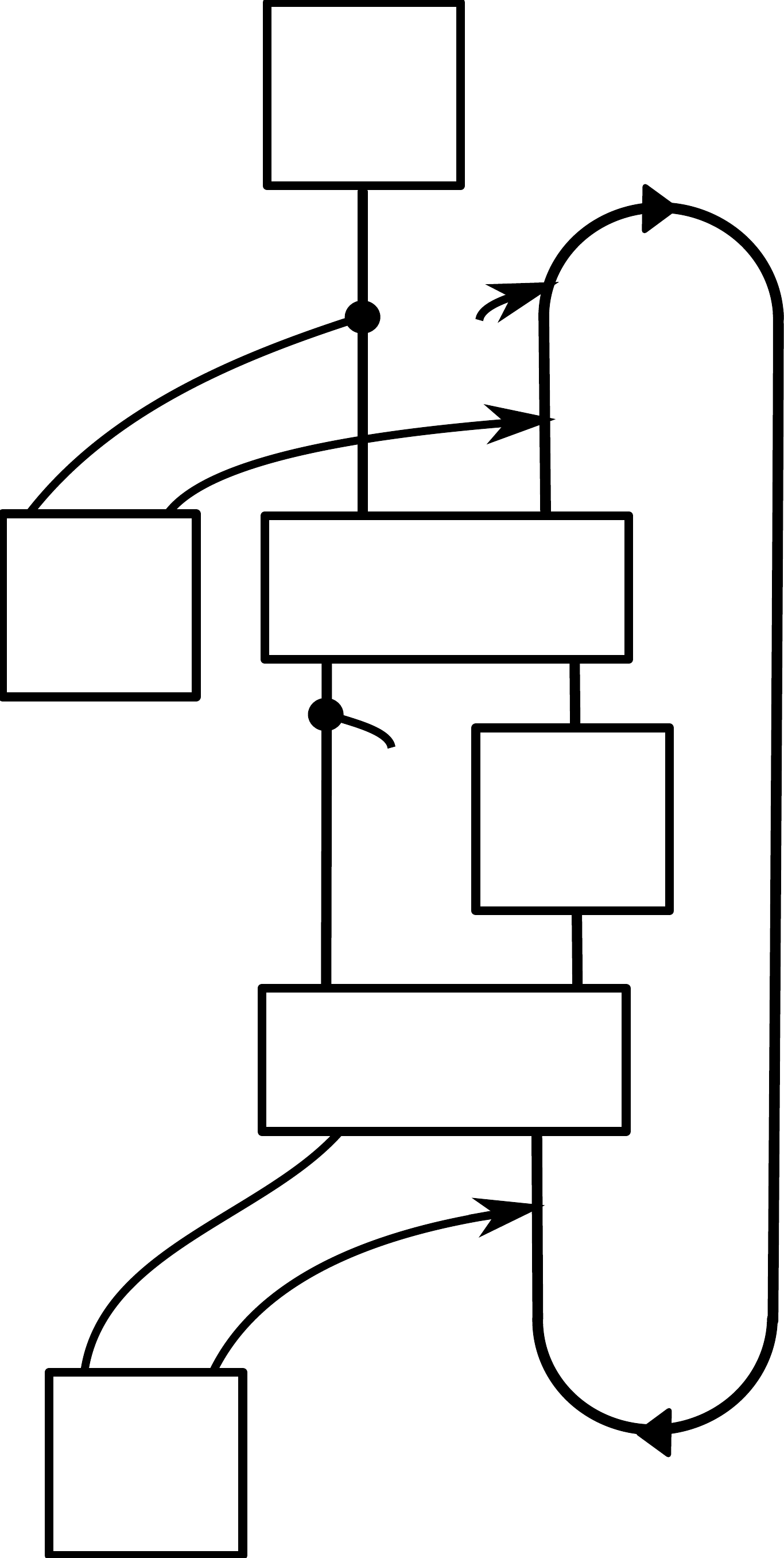}{0.3}
                    \put(-67,98){\scalebox{1.3}{$t$}}
                    \put(-55,62){\scalebox{1.2}{$\pivotQ$}}
                    \put(-111,20){\scalebox{1.3}{$\qR$}}
                    \put(-60,23){\scalebox{1.3}{$\Fr$}}
                    \put(-59,-1){\scalebox{1.3}{$a$}}
                    \put(-40,-10){\scalebox{1.3}{$m$}}
                    \put(-60,-48){\scalebox{1.3}{$\Gr$}}
                    \put(-103,-108){\scalebox{1.3}{$\pR$}}
                }
                \ =
                \scalebox{0.8}{
                    \ipic{-0.5}{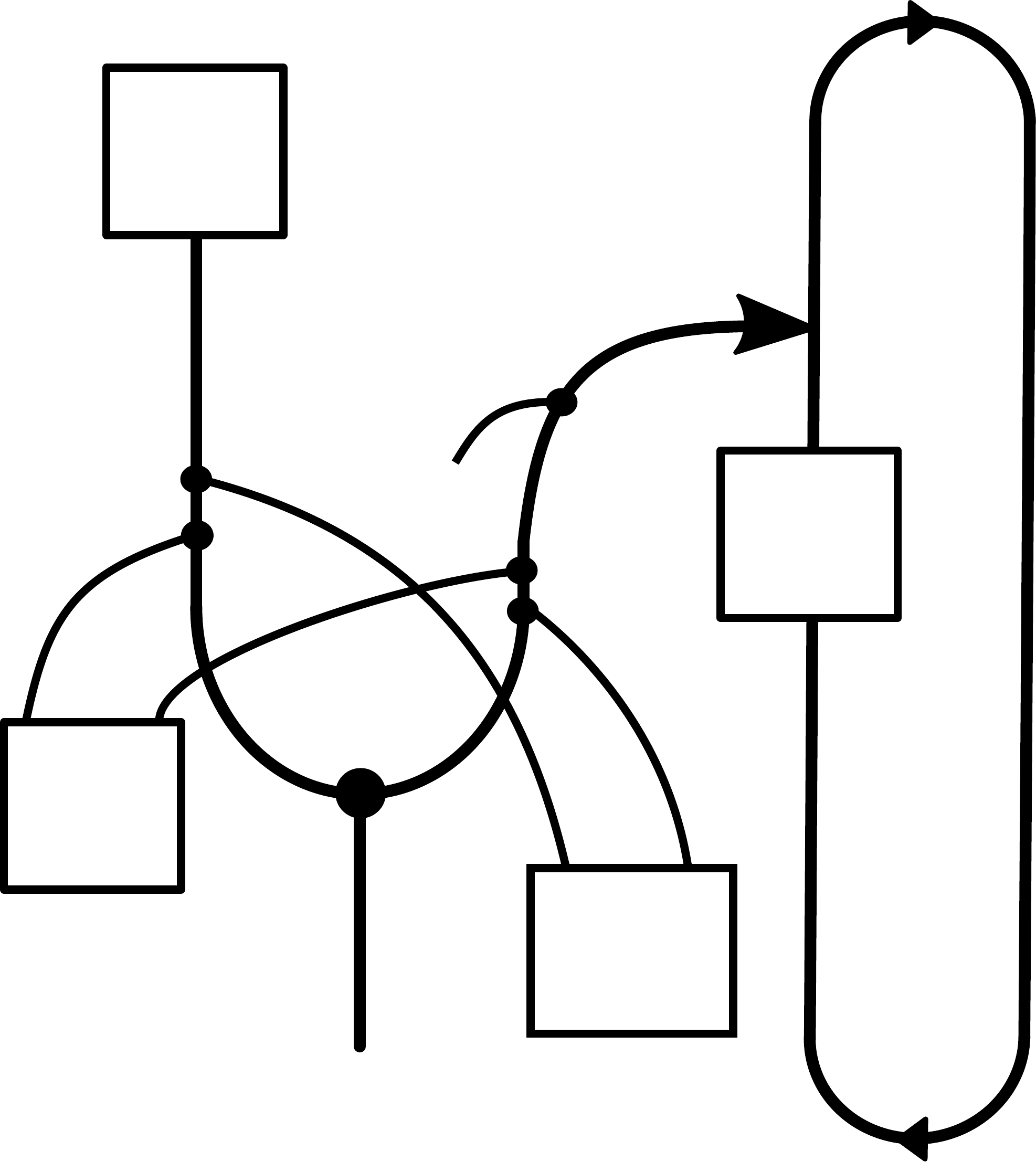}{0.3}
                    \put(-142,65){\scalebox{1.3}{$t$}}
                    \put(-101,11){\scalebox{1.2}{$\pivotQ$}}
                    \put(-162,-43){\scalebox{1.3}{$\qR$}}
                    \put(-115,-88){\scalebox{1.3}{$a$}}
                    \put(-44,04){\scalebox{1.3}{$m$}}
                    \put(-73,-67){\scalebox{1.3}{$\pR$}}
                }
            \end{equation*}
            \caption{Calculating the left hand side of 
                \eqref{proof_rptc2}.
                The string diagrams are all in $\Vect$.
                To arrive at the initial string diagram, recall the expression \eqref{eq:right-eval-explicit} for the right evaluation map.
                The first step is just substitution of $f$ from \eqref{eq:parametrise-End(HW)}.
                In the second step, we use that $\Gr(\pR_1\otimes \pR_2 w)=\oneQ \otimes w$
for all $w\in W$ (which in turn is immediate from $\Gr \circ \Fr = \id$) and substitute the definition of $\Fr$.
            }
            \label{fig:proof_rptc}
        \end{figure}
    
\smallskip

\noindent
$\modTr_H \left( \mathsf{tr}^r_W(f) \right)$: Here we use that
        $\modTr_H \left( \mathsf{tr}^r_W(f) \right)
        = 
        t\left( \textsf{tr}_W^r(f) (\oneQ) \right)$ and then rewrite
        the resulting expression as in Figure~\ref{fig:proof_rptc}. Altogether, this gives
\begin{align}\label{proof_rptc2}
	\modTr_H \left( \mathsf{tr}^r_W(f) \right)
	=
	t(\qR_1 a\sweedler{1} \pR_1) 
	\mathsf{tr}_{\field}
	\left(
	\rho(\pivotQ \qR_2 a\sweedler{2} \pR_2\otimes -)
	\circ m
	\right)\ ,
\end{align}
where $\rho:H\otimes_{\field} W \to W$ is the
action of $H$ on $W$. 

\smallskip

Since \eqref{proof_rptc} and \eqref{proof_rptc2} hold in particular for $W=H$, the left regular module, and for all $a$, $m$, we can rephrase condition \eqref{thm_redLemma_right} in Lemma~\ref{lem:redLemma} as follows: the symmetric linear form $t$ on $H$ extends to a right modified trace on $\hpmod$ if and only if
        \begin{align}
            t(a) \oneQ
            =
            (t\otimes \pivotQ)
            \left( \qR \Delta(a) \pR \right)\ .
        \end{align}

        But this is just the defining equation \eqref{eq_symmetrized_right_cointegral_version_of_modTr} 
        for a symmetrised right cointegral.
        
        \medskip

        The left version of the proof is completely analogous and uses 
        (\ref{eq_symmetrized_left_cointegral_version_of_modTr}).
        
\smallskip

\noindent
\textit{(3)}
        By Proposition \ref{prop_symmetrized_is_nondegenerate} the symmetrised 
right/left cointegrals are non-degenerate.
        It is shown in \cite[Thm.\,2.6]{BBG} that this implies that the corresponding right/left modified traces are non-degenerate in the sense of Definition~\ref{def:mod-tr}\,(ii).
    \end{proof}

\section{Example: symplectic fermion quasi-Hopf algebra}\label{Sec:SymFerm}

In this section we will use Theorem~\ref{thm:symcoint-is-modtr} to compute the modified trace for the so-called symplectic fermion quasi-Hopf algebras defined in \cite{FGR2}.
One reason that these quasi-Hopf algebras are of interest is their relation to a fundamental example of logarithmic two-dimensional conformal field theories, namely the symplectic fermion conformal field theory, see \cite{FGR2} for more details and references.

\subsection*{Quasi-Hopf structure}
The family of symplectic fermion ribbon quasi-Hopf algebras $\symFerm=\symFerm(N,\beta)$,
where $N$ is a non-zero natural number and $\beta \in \mathbb{C}$ satisfies $\beta^4=(-1)^N$, is defined as follows \cite[Sec.\,3]{FGR2}.
As a $\mathbb{C}$-algebra, $\symFerm$ is a unital associative algebra generated by
\begin{align}\label{ex_symFerm_generators}
    \{\,\genK, \genF_i^\epsilon \ |\ 1\leq i \leq N,~ 
    \epsilon= \pm \,\}\ .
\end{align}
With the elements 
\begin{align}
    \eQ_0  = \tfrac{1}{2}(\oneQ + \genK^2)
    \ , \qquad
    \eQ_1  = \tfrac{1}{2}(\oneQ - \genK^2)
\end{align}
we can write the defining relations for $\symFerm$ as
\begin{align}
    \{\genF^{\pm}_i,\genK \} = 0\ ,
    \quad
    \{\genF^+_i, \genF^-_j \} = \delta_{i,j} \eQ_1\ ,
    \quad
    \{\genF^\pm_i, \genF^\pm_j \} = 0\ ,
    \quad
    \genK^4 = \oneQ\ ,
\end{align}
where $\{-,-\}$ is the anticommutator. 
Then 
$\eQ_0, \eQ_1$ are central orthogonal idempotents with 
$\eQ_0 + \eQ_1 = \oneQ$. 
The dimension of $\symFerm$ is $2^{2N+2}$.

It is enough to specify the quasi-Hopf algebra structure on generators. 
The coproduct is
\begin{align}
    \Delta(\genK) 
    &= \genK \otimes \genK - (1+(-1)^N)\ \eQ_1 \genK \otimes \eQ_1 \genK \ ,
    \notag\\ 
    \Delta(\genF^\pm_i)
    &= \genF^\pm_i \otimes \oneQ + \omega_\pm \otimes \genF^\pm_i \ ,
\end{align}
where $\omega_\pm = (\eQ_0 \pm i \eQ_1) \genK$. 
The counit is 
\begin{align}
    \counit(\genK) = 1 \ ,
    \qquad 
    \counit(\genF^\pm_i) = 0\ .
\end{align}
We introduce
\begin{align}
    \betaQ_\pm = \eQ_0 + \beta^2 (\pm i\genK)^N \eQ_1
\end{align}
to define the coassociator and its inverse as
\begin{align}
    \coassQ^{\pm 1} = 
    \oneQ \otimes \oneQ \otimes \oneQ
    + \eQ_1 \otimes \eQ_1 \otimes 
    \left\{
        \eQ_0(\genK^N - \oneQ)
        + \eQ_1(\betaQ_\pm  - \oneQ) 
    \right\}\ .
\end{align}
Finally, the antipode $S$ and the evaluation and coevaluation elements 
$\alphaQ$ and $\betaQ$ are given by
\begin{alignat}{2}
    S(\genK) &= \genK^{(-1)^N} = (\eQ_0 + (-1)^N \eQ_1)\genK \ ,
    \qquad\qquad && \alphaQ = \oneQ \ ,
    \notag\\ 
    S(\genF^\pm_i) &= \genF^\pm_i (\eQ_0 \pm (-1)^N i\eQ_1) \genK \ , 
    && \betaQ = \betaQ_+\ .
\end{alignat}
For convenience we also state the inverse antipode on generators:
\begin{align}
    S\inv(\genK) = \genK^{(-1)^N} \ , \qquad
    S\inv(\genF^\pm_i) = \omega_{\pm} \genF^\pm_i\ .
\end{align}
Note that $S(\betaQ_\pm)=S\inv(\betaQ_\pm)=\betaQ_\mp$, and $\betaQ_+\betaQ_-=\oneQ$.

The pivot of $\symFerm$ is\footnote{
	The symbol $\pivotQ$ has a slightly different meaning in \cite{FGR2}, and so the expression for $\pivotQ$ stated there differs from the one given here.}
\begin{align}\label{eq:SF-pivot}
	\pivotQ 
    = (\eQ_0 + (-i)^{N+1} \eQ_1 \genK^N )\genK\ .
\end{align}
>From \cite[Eq.\,(3.35)]{FGR2} we know that the 
Drinfeld twist and its inverse are given by 
\begin{align}
	\Dt^{\pm 1} 
	= \eQ_0 \otimes \oneQ 
	+ \eQ_1 \otimes \eQ_0 \genK^N 
	+ \eQ_1 \betaQ_\mp \otimes \eQ_1 \ .
\end{align}
One furthermore computes
\begin{align}
	\qR &= \oneQ \otimes \oneQ + \eQ_1 \otimes\big(  \eQ_1 
	( \betaQ - \oneQ ) \big) \ ,
	\notag\\ 
	\pR &= 
	\oneQ\otimes \oneQ + \eQ_0\otimes \big(\eQ_1(\betaQ - \oneQ )\big) 
	\ , 
	\notag\\ 
	\qL &= \oneQ \otimes \oneQ 
	+ \eQ_1 \otimes 
	\left\{
	\eQ_0(\genK^N - \oneQ)
	+ \eQ_1(\betaQ - \oneQ) 
	\right\}
	\ , 
	\notag\\ 
	\pL &= \betaQ_- \otimes \oneQ
	+ \eQ_1 \betaQ_- \otimes 
	\left\{
	\eQ_0(\genK^N - \oneQ)
	+ \eQ_1(\betaQ_- - \oneQ)     
	\right\}\ .
\end{align}

\medskip 

For later use, we fix a basis of $\symFerm$. 
The basis elements are
\begin{align}
	\basisEl{a}{b}{i}
	\vcentcolon=
	\left(
	\prod_{j=1}^{|\vec{a}|}
	\genF^+_{a_j}
	\right)
	\left(
	\prod_{k=1}^{|\vec{b}|}
	\genF^-_{b_k}
	\right)
	\genK^i\ ,
\end{align}
where $i\in \mathbb{Z}_4$, and $\vec{a},\vec{b}$ are strictly ordered multi-indices 
of lengths $0\leq |\vec{a}|,|\vec{b}| \leq N$. 
By ``strictly ordered'' we mean that for $\vec a = (a_1,a_2,\dots,a_{|\vec a|})$ we have
$1\leq a_1 < \ldots < a_{|\vec{a}|}\leq N$,
and similarly for $\vec{b}$.
The element corresponding to $\basisEl{a}{b}{i}$ in the dual basis is denoted by
\begin{align}
    \big( \basisEl{a}{b}{i} \big) ^*.
\end{align}
We will use the shorthand
\begin{align}
	\vec{N}=(1,2,\ldots,N) \ .
\end{align}

Using this notation  we can state that
\begin{align}
	\intQ = \sum_{j=0}^3 \basisEl{N}{N}{j}
\end{align}
is both a left and a right integral in $\symFerm$ \cite[Sec.\,3.5]{FGR2}.
In particular, $\symFerm$ is unimodular.

\medskip

The quasi-Hopf algebra $\symFerm$ can be equipped with an $R$-matrix and a ribbon element, turning it into a ribbon quasi-Hopf algebra.
In \cite[Prop.\,3.2]{FGR2} it was shown that it is in fact a factorisable ribbon quasi-Hopf algebra.
Factorisability implies unimodularity \cite[Sec.\,6]{BT1}, giving another argument showing that $\symFerm$ is unimodular.
A ribbon category is in particular pivotal. The pivot in \eqref{eq:SF-pivot} was obtained as $\pivotQ = \ribbon\inv \drinfeldElement$, where $\ribbon$ is the ribbon element and $\drinfeldElement$ is the Drinfeld element.

\subsection*{Modified trace}
We will see that the spaces of left and right modified traces coincide for $\symFerm$.
To compute the modified trace explicitly, we first find the (also coinciding) left and
right symmetrised cointegrals via Corollary~\ref{coro_symmetrized_simple_condition}. Then we employ Theorem~\ref{thm:symcoint-is-modtr} and the relation~\eqref{eq:symCoint-to-modTr} to obtain the value of the modified trace on the projective generator $\symFerm$.

\begin{prop}\label{prop_symmetrized_cointegral_for_symFerm}
    The linear form 
	\begin{align}\label{eq:sf-symcoint}
	\cointSymR = (\beta^2 + i)\cdot \left( \basisEl{N}{N}{1} \right)^*
	+ (\beta^2 - i)\cdot \left( \basisEl{N}{N}{3} \right)^*
	\end{align}
	is simultaneously a left and a right symmetrised cointegral for $\symFerm$.
\end{prop}
\begin{proof}
We will verify that $\cointSymR$ satisfies both conditions in 
Corollary~\ref{coro_symmetrized_simple_condition}.
To this end, we first note that the coproduct takes the following form on elements of the above basis:
\begin{align}\label{symFerm_coproduct_basis}
\Delta(\basisEl{l}{a}{i})
=& \left( \basisEl{l}{a}{i} \otimes \genK^i + 
\omega_+^{|\vec{a}|} \omega_-^{|\vec{b}|}\genK^i \otimes \basisEl{l}{a}{i} 
+ \text{ (lower terms)}
\right) \notag\\ 
&\times 
\left(
\eQ_0 \otimes \eQ_0 +
\eQ_0 \otimes \eQ_1 +
\eQ_1 \otimes \eQ_0 +
(-1)^{(N+1)i} \eQ_1 \otimes \eQ_1
\right) \ ,
\end{align}
where in each tensor factor in ``(lower terms)'' the number of $\genF^+$'s is strictly 
less than $|\vec{a}|$, or the number of $\genF^-$'s is strictly less than $|\vec{b}|$, or both.
Therefore, both sides of the two conditions in
Corollary~\ref{coro_symmetrized_simple_condition} vanish identically unless one chooses $h
= \basisEl{N}{N}{i}$, $i \in \{0,1,2,3\}$. In these four cases a straightforward
computation shows that the conditions in Corollary~\ref{coro_symmetrized_simple_condition} hold.
\end{proof}

Note that because $\beta^4 = (-1)^N$,
for odd $N$ only one of the two summands in \eqref{eq:sf-symcoint} is present, the other coefficient is zero.
For $N$ even, both summands are present.

Since the symmetrised cointegral is two-sided, so is the corresponding modified trace.
By~\eqref{eq:symCoint-to-modTr} the explicit value of the modified trace on $f \in \End_\symFerm(\symFerm)$ is
\begin{align}
    \modTr_\symFerm(f) = \widehat{\coint}(f(\one)) \ ,
\end{align}
with $\widehat{\coint}$ as in \eqref{eq:sf-symcoint}. 

\medskip

The modified trace has also been computed by a different method in \cite[Sec.\,9]{GR}, namely by using the existence of a simple projective object in $\symFerm\text{-mod}$.
There, the modified trace is given on the four indecomposable projectives.
To relate the two computations, first note that the central idempotents of $\symFerm$ are
\begin{align}
	\eQ_0 
	\ , \qquad 
	\eQ_1^\pm = \tfrac12 \eQ_1 (\oneQ \pm \beta^{-1} \ribbon)
	= \tfrac12 \eQ_1 \Big( \oneQ \mp i \genK 
    \prod_{k=1}^N (\oneQ - 2 \genF^+_k \genF^-_k) \Big) \ ,
\end{align}
see~\cite[Sec.\,3.6]{FGR2}.
The decomposition of the right regular module $\symFerm$ is
\begin{align}
\symFerm = P_{0+} \oplus P_{0-} \oplus X_{1+}^{\oplus 2^N} \oplus X_{1-}^{\oplus 2^N}  \ ,
\end{align}
where $P_{0\pm}$ are the projective covers of the two one-dimensional simple modules of $\symFerm$ and $X_{1\pm}$ are projective simple objects of dimension $2^N$ \cite[Sec.\,3.7]{FGR2}.
The projections to $P_{0\pm}$ are given by right-multiplication with the (non-central) idempotents $\eQ_0^\pm = \tfrac12 (\one \pm \genK)\eQ_0$.
The central idempotents $\eQ_1^\pm$ project to the direct sums $X_{1\pm}^{\oplus 2^N}$.
Set
\begin{align}
x_\pm = \left(\prod_{j=1}^N \genF^+_j \genF^-_j\right) \eQ_0^\pm \ ,
\qquad
y_\pm = \eQ_1^\pm \ .
\end{align}
Note that $x_\pm$ and $y_\pm$ are central in $\symFerm$ \cite[Sec.\,3.6]{FGR2}.
It is straightforward to compute the modified trace of $r_{x_{\pm}}, r_{y_{\pm}} \in \End_{\symFerm}(\symFerm)$:
\begin{align}
\modTr_\symFerm(r_{x_{\pm}}) = \pm \tfrac12 (-1)^{\frac12 N(N-1)} \beta^2 \ ,
\quad
\modTr_\symFerm(r_{y_{\pm}}) = \pm \tfrac12 (-1)^{\frac12 N(N-1)} (-2)^N \ ,
\end{align}
where $r_h$ denotes the right multiplication with $h\in \symFerm$,
cf.~\eqref{eq:left-right-mult}.
This agrees with \cite[Sec.\,9]{GR} up to a normalisation factor of $\tfrac12 (-1)^{\frac12 N(N+1)}$.

\medskip

Since $\pivotQ$ is of order two, the left and right cointegrals also agree. One can compute 
the cointegral for $\symFerm$ by shifting 
the symmetrised cointegral from Proposition~\ref{prop_symmetrized_cointegral_for_symFerm}
by $\pivotQ$.
Similar to the symmetrised cointegral, it is non-vanishing only on the top components, and with 
\begin{alignat}{2}
    a_\pm &= \beta^2 \pm \delta_{N,\text{even}}\ 
    \quad ,\quad
    b_\pm &&= \pm i \delta_{N,\text{odd}}
\end{alignat}
it can be expressed as
    \begin{align}
        \coint = 
        a_+ \left( \basisEl{N}{N}{0} \right)^*
        +
        b_+ \left( \basisEl{N}{N}{1} \right)^*
        +
        a_- \left( \basisEl{N}{N}{2} \right)^*
        +
        b_- \left( \basisEl{N}{N}{3} \right)^*\ .
    \end{align}

\newcommand\arxiv[2]      {\href{http://arXiv.org/abs/#1}{#2}}
\newcommand\doi[2]        {\href{http://dx.doi.org/#1}{#2}}


\begin{thebibliography}{GKP2}
	\bibitem[BBGa]{BBG}
    A.~Beliakova, C.~Blanchet, A.M.~Gainutdinov,
    {\it Modified trace is a symmetrised integral},
    \arxiv{1801.00321}{[1801.00321 [math.QA]]}.

    \bibitem[BBGe]{BBGe}
    A.~Beliakova, C.~Blanchet, N.~Geer,
    {\it Logarithmic Hennings invariants for restricted quantum sl(2)}
    \doi{10.2140/agt.2018.18.4329}{Algebr.\ Geom.\ Topol.\ {\bf 18} (2018), 4329--4358},
    \arxiv{1705.03083}{[1705.03083 [math.GT]]}.

	\bibitem[BC1]{BC1}
	D.~Bulacu, S.~Caenepeel,
	{\it Integrals for (dual) quasi-Hopf algebras. Applications},
	\doi{10.1016/S0021-8693(03)00175-3}{J.\ Algebra {\bf 266} (2003), 552--583},
	\arxiv{math/0110063}{[math/0110063 [math.QA]]}.

	\bibitem[BC2]{BC2}
	D.~Bulacu, S.~Caenepeel,
	{\it On integrals and cointegrals for quasi-Hopf algebras},
	\doi{10.1016/j.jalgebra.2011.11.006}{J.\ Algebra {\bf 351} (2012), 390--425},
	\arxiv{1103.2263}{[1103.2263 [math.QA]]}.

	\bibitem[BCT]{BCT}
	D.~Bulacu, S.~Caenepeel, B.~Torrecillas,
	{\it Involutory quasi-Hopf algebras},
	\doi{10.1007/s10468-009-9142-9}{Algebr.\ Represent.\ Theory {\bf 12} (2009), 257--285},
	\arxiv{0704.3036}{[0704.3036 [math.QA]]}.

    \bibitem[BGR]{BGR2}
    J.~Berger, A.M.~Gainutdinov, I.~Runkel,
    {\it Monadic cointegrals and applications to quasi-Hopf algebras}, in preparation.

	\bibitem[BT1]{BT1}
	D.~Bulacu, B.~Torrecillas,
	{\it Factorizable quasi-Hopf algebras. Applications},
	\doi{10.1016/j.jpaa.2004.04.010}{J.\ Pure Appl.\ Algebra {\bf 194} (2004), 39--84},
	\arxiv{math/0312076}{[math/0312076 [math.QA]]}.

	\bibitem[BT2]{BT2}
	D.~Bulacu, B.~Torrecillas,
	{\it On sovereign, balanced and ribbon quasi-Hopf algebras},
	\arxiv{1811.11628}{[1811.11628 [math.QA]]}.

    \bibitem[CG]{CG} 
    T.~Creutzig, T.~Gannon, 
    {\it Logarithmic conformal field theory, log-modular tensor categories and modular forms}, 
    \doi{10.1088/1751-8121/aa8538}{J.\ Phys.\ A {\bf 50} (2017) 404004},
    \arxiv{1605.04630}{[1605.04630 [math.QA]]}.

    \bibitem[CGP]{CGPM}
    F.~Costantino, N.~Geer, B.~Patureau-Mirand,
    {\it Some remarks on the unrolled quantum group of sl(2)},
    \doi{10.1016/j.jpaa.2014.10.012}{J.\ Pure and Appl.\ Algebra {\bf 219} (2015), 3238--3262},
    \arxiv{1406.0410}{[1406.0410 [math.QA]]}.

    \bibitem[DGP]{DRGPM}
    M.~De~Renzi, N.~Geer, B.~Patureau-Mirand,
    {\it Renormalized Hennings invariants and 2+1-TQFTs},
    \doi{10.1007/s00220-018-3187-8}{Commun.\ Math.\ Phys.\ {\bf 362} (2018), 855--907},
    \arxiv{1707.08044}{[1707.08044 [math.GT]]}.

	\bibitem[Dr]{Dr}
    V.~Drinfeld,
    {\it Quasi-Hopf algebras},
    {Leningrad Math.\ J.\ {\bf 1} (1990), 1419--1457}.

	\bibitem[ENO]{ENO}
	P.~Etingof, D.~Nikshych, V.~Ostrik,
	{\it An analogue of Radford’s $S^4$ formula for finite tensor categories},
	\doi{10.1155/S1073792804141445}{IMRN {\bf 52} (2004), 2915--2933},
	\arxiv{math/0404504}{[math/0404504 [math.QA]]}.

	\bibitem[FGR1]{FGR1}
	V.~Farsad, A.M.~Gainutdinov, I.~Runkel,
	{\it $SL(2,\mathbb{Z})$-action for ribbon quasi-Hopf algebras},
    \doi{10.1016/j.jalgebra.2018.12.012}{J.\ Algebra {\bf 522} (2019), 243--308},
	\arxiv{1702.01086}{[1702.01086 [math.QA]]}.

	\bibitem[FGR2]{FGR2}
	V.~Farsad, A.M.~Gainutdinov, I.~Runkel,
	{\it The symplectic fermion ribbon quasi-Hopf algebra and the $SL(2,\mathbb{Z})$-action on its centre},
	\arxiv{1706.08164}{[1706.08164 [math.QA]]}.

    \bibitem[FOG]{FOG}
    A.F.~Fontalvo Orozco, A.M.~Gainutdinov,
    {\it Module traces and Hopf group-coalgebras},
    \arxiv{1809.01122}{1809.01122 [math.QA]}.

	\bibitem[GPV]{GPMV}
    N.~Geer, B.~Patureau-Mirand, A.~Virelizier
	{\it Traces on ideals in pivotal categories},
    \doi{10.4171/QT/36}{Quantum Topol.\ {\bf 4} (2013), 91--124},
    \arxiv{1103.1660}{[1103.1660 [math.QA]]}.

    \bibitem[GKP1]{GKP-lastaddition}
	N.~Geer, J.~Kujawa, B.~Patureau-Mirand,
	{\it Generalized trace and modified dimension functions on ribbon categories},
    \doi{10.1007/s00029-010-0046-7}{B.\ Sel.\ Math.\ New Ser.\ {\bf 17} (2011), 453--404},
    \arxiv{1001.0985}{[1001.0985 [math.RT]]}.

	\bibitem[GKP2]{GKPM1}
	N.~Geer, J.~Kujawa, B.~Patureau-Mirand,
	{\it Ambidextrous objects and trace functions for nonsemisimple categories},
    \doi{10.1090/S0002-9939-2013-11563-7}{Proc.\ Amer.\ Math.\ Soc.\ {\bf 141} (2013), 2963--2978},
    \arxiv{1106.4477}{[1106.4477 [math.RT]]}.

	\bibitem[GKP3]{GKPM2}
	N.~Geer, J.~Kujawa, B.~Patureau-Mirand,
	{\it M-traces in (non-unimodular) pivotal categories},
	\arxiv{1809.00499}{[1809.00499 [math.RT]]}.

	\bibitem[GPT]{GPT}
	N.~Geer, B.~Patureau-Mirand, V.~Turaev,
	{\it Modified quantum dimensions and re-normalized link invariants},
	\doi{10.1112/S0010437X08003795}{Compositio Math.\ {\bf 145} (2009), 196--212},
	\arxiv{0711.4229}{[0711.4229 [math.QA]]}.

    \bibitem[GR1]{GR1}
    A.M.~Gainutdinov, I.~Runkel, 
    {\it The non-semisimple Verlinde formula and pseudo-trace functions}, 
    \doi{10.1016/j.jpaa.2018.04.014}{J.\ Pure Appl.\ Alg.\ {\bf 223} (2019), 660--690},
    \arxiv{1605.04448}{[1605.04448 [math.QA]]}.

	\bibitem[GR2]{GR}
	A.M.~Gainutdinov, I.~Runkel,
    {\it Projective objects and the modified trace in factorisable finite tensor
    categories},
	\arxiv{1703.00150}{[1703.00150 [math.QA]]}.

    \bibitem[Ha]{Ha}
    N.-P.~Ha,
    {\it Modified trace from pivotal Hopf G-coalgebra},
    \arxiv{1804.02416}{1804.02416 [math.QA]}.

	\bibitem[HN1]{HN-doubles}
	F.~Hausser, F.~Nill,
	{\it Doubles of quasi-quantum groups},
	\doi{10.1007/s002200050512}{Comm.\ Math.\ Phys.\ {\bf 199} (1999), 547--589}.

	\bibitem[HN2]{HN-integrals}
	F.~Hausser, F.~Nill,
	{\it Integral theory for quasi-Hopf algebras},
	\arxiv{math/9904164}{math/9904164 [math.QA]}.

    \bibitem[Ra]{Radford}
    D.~E.~Radford,
    {\it Finiteness conditions for a Hopf algebra with a nonzero integral},
    \doi{10.1016/0021-8693(77)90400-8}{J.\ Algebra {\bf 46} (1977) 189--195}.

	\bibitem[Sch]{Sch}
	P.~Schauenburg,
	{\it Hopf modules and the double of a Quasi-Hopf Algebra},
    \href{https://www.jstor.org/stable/3073046}{Trans.\ Amer.\ Math.\ Soc.\ {\bf 354}
    (2002), 3349--3378}.

    \bibitem[Se]{Selinger}
    P.~Selinger,
    {\it A survey of graphical languages for monoidal categories},
    in:
    \doi{10.1007/978-3-642-12821-9_4}{New Structures for Physics (2010), 289--255},
    \arxiv{0908.3347}{[0908.3347 [math.CT]]}.

    \bibitem[SS]{SS}
    T.~Shibata, K.~Shimizu,
    {\it Categorical aspects of cointegrals on quasi-Hopf algebras},
    \arxiv{1812.03439}{[1812.03439 [math.QA]]}.

\end{thebibliography}
\end{document}